\documentclass[10pt, a4paper, reqno, english]{amsart}

\usepackage{hyperref}
\usepackage{amssymb}
\usepackage[all]{xy}

\newcommand\CC{\mathbb{C}}
\newcommand\PP{\mathbb{P}}
\newcommand\AAA{\mathbb{A}}
\newcommand\QQ{\mathbb{Q}}
\newcommand\GG{\mathbb{G}}
\newcommand\xx{\mathbf{x}}
\newcommand{\ADE}{\mathbf{ADE}}
\newcommand{\Aone}{{\mathbf A}_1}
\newcommand{\Athree}{{\mathbf A}_3}
\newcommand{\Afour}{{\mathbf A}_4}
\newcommand{\Dfour}{{\mathbf D}_4}
\newcommand{\Dfive}{{\mathbf D}_5}
\newcommand{\tS}{{\widetilde S}}
\newcommand{\abs}[1]{\left\|#1\right\|_\infty}
\newcommand{\ee}{\boldsymbol{\eta}}
\newcommand{\e}{\eta}
\newcommand\dd{\,\mathrm{d}}
\newcommand{\ex}[1]{*+<5pt>[o][F]{E_{#1}}}
\newcommand{\li}[1]{*+<3pt>[F]{E_{#1}}}
\newcommand{\classrep}{\mathcal{C}}
\newcommand{\classtuple}{\mathbf{C}}
\newcommand{\OO}{\mathcal{O}}
\newcommand{\eI}{I}
\newcommand{\eII}{\mathbf{I}}
\newcommand{\N}{\mathfrak{N}}
\newcommand{\id}[1]{\mathfrak{#1}}
\newcommand{\aaa}{\id a}
\newcommand{\p}{\id p}
\newcommand{\kc}{\id{k}_\id{c}}
\newcommand{\Ao}{\Pi_1}
\newcommand{\At}{\Pi_2}
\newcommand{\eeA}{\boldsymbol{\eta}_{\boldsymbol{A}}}
\newcommand\rto{\dashrightarrow}

\newtheorem{theorem}{Theorem}
\newtheorem{lemma}[theorem]{Lemma}

\theoremstyle{definition}
\newtheorem*{ack}{Acknowledgements}

\numberwithin{theorem}{section}
\numberwithin{equation}{section}

\begin{document}

\setcounter{tocdepth}{1}

\title{Counting imaginary quadratic points via universal torsors, II}

\author{Ulrich Derenthal}

\address{Mathematisches Institut, Ludwig-Maximilians-Universit\"at M\"unchen,
  Theresienstr. 39, 80333 M\"unchen, Germany}

\email{ulrich.derenthal@mathematik.uni-muenchen.de}

\author{Christopher Frei}

\address{Institut f\"ur Mathematik A, Technische Universit\"at Graz,
  Steyrergasse 30, 8010 Graz, Austria}

\email{frei@math.tugraz.at}

\date{April 11, 2013}

\begin{abstract}
  We prove Manin's conjecture for four singular quartic del Pezzo surfaces
  over imaginary quadratic number fields, using the universal torsor method.
\end{abstract}


%
%

\maketitle

\tableofcontents

\section{Introduction}

Let $K$ be a number field, $S$ a del Pezzo surface defined over $K$ with only
$\ADE$-singularities, $U$ the open subset obtained by removing the lines from
$S$, and $H$ a height function on $S$ coming from an anticanonical
embedding. If $S(K)$ is Zariski dense in $S$ then generalizations
(e.g.\ \cite{MR1679843}) of Manin's conjecture \cite{MR89m:11060, MR1032922}
predict an asymptotic formula, as $B \to \infty$, for the quantity
\begin{equation*}
  N_{U,H}(B) := |\{\xx \in U(K) \mid H(\xx) \le B\}|,
\end{equation*}
namely
\begin{equation*}
  N_{U,H}(B) = c_{S,H}B(\log B)^{\rho-1}(1+o(1)),
\end{equation*}
where $\rho$ is the rank of the Picard group of a minimal desingularization of
$S$ and $c_{S,H}$ is a positive real number.

Much progress was made in recent years in proving Manin's conjecture for
specific del Pezzo surfaces over $\QQ$ via the \emph{universal torsor
  method}. In \cite{arXiv:1302.6151}, the authors extended this method to
imaginary quadratic fields in case of a quartic del Pezzo surface of type
$\Athree$ with five lines.

In the present article, we continue this investigation by proving Manin's
conjecture over imaginary quadratic fields for quartic del Pezzo
surfaces of types $\Athree+\Aone$, $\Afour$, $\Dfour$, and $\Dfive$.

For more information about Manin's conjecture and the universal torsor method,
we refer to the introductory section of \cite{arXiv:1302.6151} and the
references mentioned there.

\subsection{Results}
Let $K$ be an imaginary quadratic field. We define the anticanonically embedded quartic del
Pezzo surfaces $S_i \subset \PP^4_K$ over $K$ by the following equations:
 \begin{align}
   S_0 &:& x_0x_1-x_2x_3=x_0x_3+x_1x_3+x_2x_4 = 0 &\quad\text{ of type
   }\Athree\text{ ($5$ lines)}\label{eq:def_A3}\\
   S_1 &:& x_0x_3-x_2x_4=x_0x_1+x_1x_3+x_2^2 = 0 &\quad\text{ of type }\Athree+\Aone,\label{eq:def_A3+A1}\\
   S_2 &:& x_0 x_1 - x_2 x_3 = x_0 x_4 + x_1 x_2 + x_3^2 = 0 &\quad\text{ of type }\Afour,\label{eq:def_A4}\\
   S_3 &:& x_0 x_3 - x_1 x_4 = x_0 x_1 + x_1 x_3 + x_2^2 = 0 &\quad\text{ of type }\Dfour,\label{eq:def_D4}\\
   S_4 &:& x_0 x_1 - x_2^2 = x_3^2 + x_0 x_4 + x_1 x_2 = 0 &\quad\text{ of type }\Dfive.\label{eq:def_D5}
 \end{align}
 All of them are split over $K$, hence rational over $K$, and therefore, their
 rational points over $K$ are Zariski dense. The Weil height on $\PP^4_K(K)$ is defined by
 \begin{equation}\label{eq:height}
   H(x_0 : \cdots : x_4) := \frac{\max\{\abs{x_0}, \dots, \abs{x_4}\}}{\N(x_0\OO_K+\dots+x_4\OO_K)}\text,
 \end{equation}
 where $\OO_K$ is the ring of integers in $K$, $\abs{\cdot} := |\cdot|^2$ is
 the square of the usual complex absolute, and $\N\aaa$ is the absolute norm
 of a fractional ideal $\aaa$.

 For $S_0$, Manin's conjecture was proved over $\QQ$ and imaginary quadratic
 fields in \cite{arXiv:1302.6151}. For $S_1$, $S_2$, $S_3$, $S_4$, Manin's
 conjecture was proved over $\QQ$ in \cite{MR2520770}, \cite{MR2543667},
 \cite{MR2290499}, \cite{MR2320172}, respectively. In this article,
 we prove Manin's conjecture for $S_1$, $\dots$, $S_4$ over imaginary
 quadratic fields:

 \begin{theorem}\label{thm:main}
   Let $K$ be an imaginary quadratic field, $\Delta_K$ its discriminant, $h_K$
   its class number, $\omega_K$ the number of units in $\OO_K$. For $i \in
   \{1, \ldots, 4\}$, let $U_i$ be the complement of the lines in the del
   Pezzo surface $S_i \subset \PP^4_K$ defined by
   \eqref{eq:def_A3+A1}--\eqref{eq:def_D5}.  For $B \geq 3$, we have
  \begin{equation*}
    N_{U_i,H}(B) = c_{S_i, H} B(\log B)^5 + O(B(\log B)^4\log \log B),
  \end{equation*}
  with
  \begin{equation*}
    c_{S_i,H} :=\alpha(\tS_i) \cdot \frac{(2\pi)^6 h_K^6}{\Delta_K^4\omega_K^6}\cdot \theta_0 \cdot \omega_{\infty}(\tS_i)\text.
  \end{equation*}
  Here,
  \begin{align*}
    \alpha(\tS_1) &:= \frac{1}{8640}, &\alpha(\tS_2) &:= \frac{1}{21600}, &
    \alpha(\tS_3) &:= \frac{1}{34560}, &\alpha(\tS_4) &:= \frac{1}{345600},
  \end{align*}
  \begin{equation}\label{eq:def_theta_0}
    \theta_0 := \prod_\p \left(1-\frac{1}{\N\p}\right)^6\left(1+\frac{6}{\N\p}+\frac{1}{\N\p^2}\right),
  \end{equation}
  and
  \begin{align*}
    \omega_\infty(\tS_1) &:= \frac{12}{\pi}\int_{\abs{z_0z_1(z_0+z_2)},
      \abs{z_1^3}, \abs{z_1^2(z_0+z_2)}, \abs{z_1z_2(z_0+z_2)}, \abs{z_0z_2(z_0+z_2)}\leq 1}\hspace{-2cm}\dd z_0 \dd z_1 \dd z_2,\\
    \omega_\infty(\tS_2) &:= \frac{12}{\pi}\int_{\abs{z_0^3}, \abs{z_0z_2z_3}, \abs{z_0^2z_2}, \abs{z_0^2z_3}, \abs{z_3(z_2^2 + z_0z_3)}\leq 1}\dd z_0 \dd z_1 \dd z_2,\\
    \omega_\infty(\tS_3) &:= \frac{12}{\pi}\int_{\abs{z_0 z_1^2}, \abs{z_1^3}, \abs{z_1^2z_2}, \abs{z_1(z_0 z_1 + z_2^2)}, \abs{z_0(z_0 z_1 + z_2^2)}\leq 1}\dd z_0 \dd z_1 \dd z_2 ,\\
    \omega_\infty(\tS_4) &:= \frac{12}{\pi}\int_{\abs{z_0^3}, \abs{z_0z_1^2}, \abs{z_0^2z_1}, \abs{z_0^2z_2}, \abs{z_0z_2^2 + z_1^3}\leq 1}\dd z_0 \dd z_1 \dd z_2 .
  \end{align*}
\end{theorem}

We note that Manin's conjecture for $S_4$ is implied by \cite{MR1906155} over
arbitrary number fields, since $S_4$ is an equivariant compactification of
$\GG_a^2$. On the other hand, $S_0, \dots, S_3$ are neither toric nor
equivariant compactifications of $\GG_a^2$ \cite{MR2753646}, so that
\cite{MR1620682, MR1906155} do not apply. Finally, $S_1$ and $S_3$ (but not
$S_0$, $S_2$, $S_4$) are equivariant compactifications of some semidirect
products $\GG_a \rtimes \GG_m$ \cite{arXiv:1212.3518}, so similar methods as
in \cite{MR1620682, MR1906155} may apply to them, but this has been worked out
only over $\QQ$ and with further restrictions in \cite{MR2858922}.

\subsection{Methods}
The general strategy in our proofs of Theorem \ref{thm:main} for $S_1$,
$\dots$, $S_4$ is the one proposed in \cite{arXiv:1302.6151}:

In a first step, the rational points $S_i(K)$ are parameterized by integral
points on universal torsors over $S_i$, satisfying certain \emph{height
  conditions} and \emph{coprimality conditions}, following the strategy from
\cite[Section 4]{arXiv:1302.6151}. Since the Cox rings of all minimal
desingularizations $\tS_i$ have only one relation \cite{math.AG/0604194}, the
universal torsors are open subsets of hypersurfaces in $\AAA^9_K$, with
coordinates $(\e_1, \dots, \e_9)$ and one relation, the \emph{torsor
  equation}.

In the second step, we approximate the number of these integral points on
universal torsors subject to height and coprimality conditions by an
integral. In all cases $\e_9$ appears linearly in the torsor equation, so it
is uniquely defined by $\e_1, \ldots, \e_8$. We first count pairs $(\e_8,
\e_9)$ for given $(\e_1, \ldots, \e_7)$ using the method from \cite[Section
5]{arXiv:1302.6151} and then sum the result over another variable using the
results from \cite[Section 6]{arXiv:1302.6151}. The summations over the
remaining variables are handled in all cases by a direct application of the
results of \cite[Section 7]{arXiv:1302.6151}.

In a third and final step, we show that the integrals from the second step
satisfy the asymptotic formulas from Theorem \ref{thm:main}. Here, the shape of
the effective cone of $\tS_i$ is crucial; after all, the volume of its dual
intersected with a certain hyperplane appears as $\alpha(\tS_i)$ in Peyre's
refinement \cite{MR1340296} of Manin's conjecture.

Though the proofs for $S_0$, $\ldots$, $S_4$ have many features in common, each
case has its own difficulties.

In the case of $S_0$, the first step is mostly covered by our general results
from \cite{arXiv:1302.6151}, whereas the second step requires dichotomies with
different orders of summation according to the relative size of the variables.

The first step in the case of $S_1$ is mostly covered by the general results as
well, but the second summation in the second step requires additional effort in
order to obtain sufficiently good error terms.

In the case of $S_2$, parts of the first step need to be treated individually,
and the second summation in the second step is more complicated, since $\e_8$
does not appear linearly in the torsor equation. Additionally, the second
summation requires a dichotomy similarly as in the case of $S_0$, in order to
handle the error terms.

The case of $S_3$ is probably the most simple one. Parts of the first step need
to be treated individually, but the summations in the second step go through
without additional tricks, so it just remains to bound the error terms.

Finally, in the case of $S_4$, parts of the first step need to be treated
individually, and the second summation in the second step is slightly more
complicated, since $\e_8$ does not appear linearly in the torsor equation.

\subsection{Notation}
Throughout this article, we use the notation introduced in
\cite[Section 1.4]{arXiv:1302.6151}. In particular, $\classrep$ denotes a fixed
system of integral representatives for the ideal classes of the ring of
integers $\OO_K$. Moreover, $\p$ always denotes a nonzero prime ideal of
$\OO_K$, and products indexed by $\p$ are understood to run over all such prime
ideals. We say that $x \in K$ is \emph{defined} (resp. \emph{invertible})
modulo an ideal $\aaa$ of $\OO_K$, if $v_\p(x) \geq 0$ (resp. $v_\p(x) = 0$) for
all $\p \mid \aaa$, where $v_\p$ is the usual $\p$-adic valuation. For $x,y$
defined modulo $\aaa$, we write $x
\equiv_\aaa y$ if $v_\p(x-y)\geq v_\p(\aaa)$ for all $\p \mid \aaa$.

\begin{ack}
  The first-named author was supported by grant DE 1646/2-1 of the Deutsche
  Forschungsgemeinschaft and by the Hausdorff Research Institute for
  Mathematics in Bonn which he would like to thank for the
  hospitality. The second-named author was partially supported by a research
  fellowship of the Alexander von Humboldt Foundation. This collaboration was
  supported by the Center for Advanced Studies of LMU M\"unchen.
\end{ack}

\section{The quartic del Pezzo surface of type $\Athree+\Aone$}

\subsection{Passage to a universal torsor}
Up to a permutation of the indices, we use the notation of \cite{math.AG/0604194}.
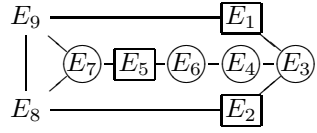
\begin{figure}[ht]
  \centering
  \[\xymatrix@R=0.05in @C=0.05in{E_9 \ar@{-}[dr] \ar@{-}[dd] \ar@{-}[rrrr]& & & & \li{1} \ar@{-}[dr]\\
    & \ex{7} \ar@{-}[r] & \li{5} \ar@{-}[r] & \ex{6} \ar@{-}[r] & \ex{4} \ar@{-}[r] & \ex{3}\\
    E_8 \ar@{-}[ur] \ar@{-}[rrrr] & & & & \li{2} \ar@{-}[ur]}\]
  \caption{Configuration of curves on $\tS_1$}
  \label{fig:A3+A1_dynkin}
\end{figure}

For any given $\classtuple = (C_0, \dots, C_5) \in \classrep^6$, we
define $u_\classtuple := \N(C_0^3 C_1^{-1}\cdots C_5^{-1})$ and
\begin{align*}
    \OO_1 &:= C_5 & \OO_2 &:= C_4 & \OO_3 &:= C_0 C_1^{-1} C_4^{-1} C_5^{-1}\\
    \OO_4 &:= C_1C_2^{-1} & \OO_5 &:= C_3 & \OO_6 &:= C_2 C_3^{-1} \\
    \OO_7 &:= C_0 C_1^{-1} C_2^{-1} C_3^{-1} & \OO_8 &:= C_0 C_4^{-1} & \OO_9 &:= C_0 C_5^{-1}.
\end{align*}
Let
\begin{equation*}
  \OO_{j*} :=
  \begin{cases}
    \OO_j^{\neq 0}, & j \in \{1,\ldots, 7\},\\
    \OO_j, & j \in \{8,9\}.
  \end{cases}
\end{equation*}
For $\eta_j \in \OO_j$, let
\begin{equation*}
  \eI_j := \e_j \OO_j^{-1}\text.
\end{equation*}
For $B \geq 0$, let $\mathcal{R}(B)$ be the set of all $(\e_1, \ldots, \e_8)
\in \CC^8$ with $\e_1 \neq 0$ and
\begin{align}
  \abs{\e_2\e_3\e_4\e_5\e_6\e_7\e_8}&\leq B,\label{eq:A3+A1_height_1}\\
  \abs{\e_1^2\e_2^2\e_3^3\e_4^2\e_6}&\leq B,\label{eq:A3+A1_height_2}\\
  \abs{\e_1\e_2\e_3^2\e_4^2\e_5^2\e_6^2\e_7} &\leq B,\label{eq:A3+A1_height_3}\\
  \abs{\e_3\e_4\e_5\e_6\e_7(\e_4\e_5^3\e_6^2\e_7 + \e_2\e_8)} &\leq B,\label{eq:A3+A1_height_4}\\
  \abs{\frac{\e_2\e_7\e_8^2 + \e_4\e_5^3\e_6^2\e_7^2\e_8}{\e_1}}&\leq
  B.\label{eq:A3+A1_height_5}
\end{align}
We observe for future reference that \eqref{eq:A3+A1_height_1}
and \eqref{eq:A3+A1_height_4} imply the condition
\begin{equation}
  \label{eq:A3+A1_height_4_useful}
  \abs{\e_3\e_4^2\e_5^4\e_6^3\e_7^2} \leq 4 B.
\end{equation}
Let $M_\classtuple(B)$ be the set of all
\begin{equation*}
  (\e_1, \ldots, \e_9) \in \OO_{1*} \times \cdots \times \OO_{9*} 
\end{equation*}
that satisfy the \emph{height conditions}
\begin{equation*}\label{eq:A3+A1_height}
  (\e_1, \ldots, \e_8) \in \mathcal{R}(u_\classtuple B)\text,
\end{equation*}
the \emph{torsor equation}
\begin{equation}\label{eq:A3+A1_torsor}
  \e_4\e_5^3\e_6^2\e_7 + \e_2\e_8 + \e_1\e_9 = 0,
\end{equation}
and the \emph{coprimality conditions}
\begin{equation}\label{eq:A3+A1_coprimality}
  \eI_j + \eI_k = \OO_K \text{ for all distinct nonadjacent vertices $E_j$, $E_k$ in Figure~\ref{fig:A3+A1_dynkin}.}
\end{equation}

\begin{lemma}\label{lem:A3+A1_passage_to_torsor}
 We have
  \begin{equation*}
    N_{U_1,H}(B) = \frac{1}{\omega_K^6}\sum_{\classtuple \in \classrep^6}|M_\classtuple(B)|.
  \end{equation*}
\end{lemma}

\begin{proof}
  We observe that the statement of our lemma is a specialization of \cite[Claim
  4.1]{arXiv:1302.6151}. We prove it using the strategy from \cite[Section
  4]{arXiv:1302.6151} based on the construction of the minimal
  desingularization $\pi : \tS_1 \to S_1$ by the following sequence of blow-ups:
  Starting with the curves $E^{(0)}_8 := \{y_0=0\}$, $E_3^{(0)} := \{y_1=0\}$,
  $E_9^{(0)} := \{y_2 = 0\}$, $E_7^{(0)} := \{-y_0 - y_2 = 0\}$ in $\PP_K^2$,
  we
  \begin{enumerate}
  \item blow up $E^{(0)}_3 \cap E^{(0)}_7$, giving $E_4^{(1)}$,
  \item blow up $E_4^{(1)} \cap E_7^{(1)}$, giving $E_6^{(2)}$,
  \item blow up $E_6^{(2)} \cap E_7^{(2)}$, giving $E_5^{(3)}$,
  \item blow up $E_3^{(3)} \cap E_8^{(3)}$, giving $E_2^{(4)}$,
  \item blow up $E_3^{(4)} \cap E_9^{(4)}$, giving $E_1^{(5)}$.
  \end{enumerate}
  With the inverse $\pi \circ \rho^{-1} : \PP^2_K \rto S_1$ of the projection
  $\phi = \rho \circ \pi^{-1} : S_1 \rto \PP^2_K$, $(x_0 :
  \cdots : x_4) \mapsto (x_0 : x_2 : x_3)$ given by
  \begin{equation}\label{eq:A3+A1_parameterization}
  \psi((y_0 : y_1 : y_2)) = (y_0y_1(y_0+y_2) : -y_1^3 : y_1^2(y_0+y_2) : y_1y_2(y_0+y_2) : y_0y_2(y_0+y_2))
  \end{equation}
  and the map $\Psi$ from \cite[Claim 4.2]{arXiv:1302.6151} sending $(\e_1,
  \ldots, \e_9)$ to 
  \begin{equation*}
    (\e_2\e_3\e_4\e_5\e_6\e_7\e_8, -\e_1^2\e_2^2\e_3^3\e_4^2\e_6 , \e_1\e_2\e_3^2\e_4^2\e_5^2\e_6^2\e_7 ,\e_1\e_3\e_4\e_5\e_6\e_7\e_9, \e_7\e_8\e_9),
  \end{equation*}
  we can proceed exactly as in the proof of \cite[Lemma
  9.1]{arXiv:1302.6151}.
\end{proof}

\subsection{Summations}
\subsubsection{The first summation over $\e_8$ with dependent $\e_9$}
\begin{lemma}\label{lem:A3+A1_first_summation}
  Write $\ee' := (\e_1, \ldots, \e_7)$ and $\eII' := (\eI_1, \ldots, \eI_7)$. For
  $B > 0$, $\classtuple \in \classrep^6$, we have
  \begin{equation*}
    |M_\classtuple(B)| = \frac{2}{\sqrt{|\Delta_K|}} \sum_{\ee' \in \OO_{1*} \times \dots \times \OO_{7*}} \theta_8(\eII')V_8(\N\eI_1, \ldots, \N\eI_7; B) + O_\classtuple(B(\log B)^2),
  \end{equation*}
  where
  \begin{equation*}
    V_8(t_1, \ldots, t_7; B) :=  \frac{1}{t_1}\int_{h(\sqrt{t_1}, \ldots, \sqrt{t_7},\e_8; B)\leq 1} \dd \e_8,
  \end{equation*}
  with a complex variable $\e_8$, and where
  \begin{equation*}
    \theta_8(\eII') := \prod_{\id p} \theta_{8,\id p}(J_{\id p}(\eII')),
  \end{equation*}
  with $J_{\id p}(\eII') := \{j \in \{1, \dots, 7\}\ :\ \id p \mid
  \eI_j\}$ and
  \begin{equation*}
    \theta_{8,\id p}(J) :=
    \begin{cases}
      1 &\text{ if }J = \emptyset\text{, }\{1\}\text{, }\{2\}\text{, }\{7\}\text,\\
      1-\frac{1}{\N\id p} &\text{ if }J = \{4\}\text{,}\{5\}\text{,}\{6\}\text{,}\{1, 3\}\text{,}\{2, 3\}\text{,}\{3, 4\}\text{,}\{4, 6\}\text{,}\{5, 6\}\text{,}\{5, 7\}\text{,}\\
      1-\frac{2}{\N\id p} &\text{ if }J = \{3\}\text,\\
      0 &\text{otherwise.}
    \end{cases}
  \end{equation*}
\end{lemma}

\begin{proof}
  By \cite[Lemma 3.2]{arXiv:1302.6151}, the set $\mathcal{R}(\ee',
  u_\classtuple B)$ of all $\e_8 \in \CC$ with $(\e_1, \ldots, \e_8) \in
  \mathcal{R}(u_\classtuple B)$ has class $m$, with an absolute constant
  $m$. Moreover, by \cite[Lemma 3.4, (1)]{arXiv:1302.6151} applied to
  \eqref{eq:A3+A1_height_5}, this set is contained in the union of at most $2$
  balls of radius
  \begin{equation*}
    R(\ee'; u_\classtuple B) :=  (u_\classtuple B\abs{\e_1 \e_2^{-1} \e_7^{-1}})^{1/4}\ll_\classtuple(B\N(\eI_1\eI_2^{-1} \eI_7^{-1}))^{1/4}
  \end{equation*}
  We apply \cite[Proposition 5.3]{arXiv:1302.6151} with $(A_1, A_2, A_3, A_0) :=
  (4, 6, 5, 7)$, $(B_1, B_0) := (2, 8)$, $(C_1, C_0) :=  (1, 9)$, $D:=3$, and
  $u_\classtuple B$ instead of $B$. (Moreover, we choose $\Ao$ and
  $\At$ as in \cite[Remark 5.2]{arXiv:1302.6151}.)

  Similarly as in \cite[Lemma 9.2]{arXiv:1302.6151}, we see that the resulting
  main term is the one given in the lemma.
  The error term from \cite[Proposition 5.3]{arXiv:1302.6151} is
  \begin{equation*}
    \ll \sum_{\ee',\ \eqref{eq:A3+A1_third_height_cond_ideals}}2^{\omega_K(\eI_3)+\omega_K(\eI_3\eI_4\eI_5\eI_6)}\left(\frac{R(\ee'; u_\classtuple B)}{\N(\eI_1)^{1/2}}+1\right)\text,
  \end{equation*}
  where, using \eqref{eq:A3+A1_height_3} and the definitions of $u_\classtuple$
  and the $\OO_j$, the sum runs over all $\ee'$ with
  \begin{equation}
    \N(\eI_1 \eI_2 \eI_3^2\eI_4^2\eI_5^2\eI_6^2 \eI_7) \leq
    B \label{eq:A3+A1_third_height_cond_ideals}.
  \end{equation}
  Since $|\OO_K^\times| < \infty$, we can sum over the $I_j$ instead of the
  $\e_j$, which then run over all nonzero ideals of $\OO_K$ with
  \eqref{eq:A3+A1_third_height_cond_ideals}, so the error term is bounded by
  \begin{align*}
      &\ll_\classtuple \sum_{\substack{\eII',\ \eqref{eq:A3+A1_third_height_cond_ideals}}}2^{\omega_K(\eI_3)+\omega_K(\eI_3\eI_4\eI_5\eI_6)}\left(\frac{B^{1/4}}{\N\eI_1^{1/4}N\eI_2^{1/4}\N\eI_7^{1/4}}+1\right)\\
      &\ll\sum_{\substack{\eI_1, \dots, \eI_6\\\N\eI_j\leq B}}\left(\frac{2^{\omega_K(\eI_3)+\omega_K(\eI_3\eI_4\eI_5\eI_6)}B}{\N\eI_1\N\eI_2\N\eI_3^{3/2}\N\eI_4^{3/2}\N\eI_5^{3/2}\N\eI_6^{3/2}}+\frac{2^{\omega_K(\eI_3)+\omega_K(\eI_3\eI_4\eI_5\eI_6)}B}{\N\eI_1\N\eI_2\N\eI_3^2\N\eI_4^2\N\eI_5^2\N\eI_6^2}\right)\\
      &\ll B(\log B)^2 + B(\log B)^2 \ll B(\log B)^2.\qedhere
  \end{align*}
\end{proof}

\subsubsection{The second summation over $\e_7$.}

\begin{lemma}\label{lem:A3+A1_second_summation}
  Write $\ee'' := (\e_1, \dots, \e_6)$. For $B \geq 3$, $\classtuple \in \classrep^6$, we have
  \begin{align*}
    |M_\classtuple(B)| &=
    \left(\frac{2}{\sqrt{|\Delta_K|}}\right)^2\sum_{\ee''\in \OO_{1*} \times \dots
      \times \OO_{6*}} \mathcal{A}(\theta_8(\eII'), \eI_7)V_7(\N\eI_1, \ldots, \N\eI_6; B)\\ &+ O_\classtuple(B(\log B)^4\log\log
    B).
  \end{align*}
  Here, for $t_1, \ldots, t_6 \geq 1$,
  \begin{equation*}
    V_7(t_1, \ldots, t_6; B) := \frac{\pi}{t_1}\int_{\substack{(\sqrt{t_1}, \ldots,
      \sqrt{t_7}, \e_8) \in \mathcal{R}(B)\\t_7\geq 1}} \dd t_7 \dd \e_8,
  \end{equation*}
  with a real variable $t_7$ and a complex variable $\e_8$.
\end{lemma}

\begin{proof}
  Following the strategy described in \cite[Section 6]{arXiv:1302.6151} in the case
  $b_0=1$, we write 
  \begin{equation}\label{eq:A3+A1_sec_sum_start}
    |M_\classtuple(B)| = \frac{2}{\sqrt{|\Delta_K|}}\sum_{\ee'' \in \OO_{1*} \times \dots \times \OO_{6*}}\sum_{\e_7 \in \OO_{7*}}\vartheta(\eI_7)g(\N\eI_7) + O_\classtuple(B(\log B)^2)\text,
  \end{equation}
  where $\vartheta(\aaa) := \theta_8(\eI_1, \ldots, \eI_6, \aaa)$ and $g(t) :=
  V_8(\N\eI_1, \ldots, \N\eI_6, t; B)$.  The conditions
  \eqref{eq:A3+A1_height_2} and \eqref{eq:A3+A1_height_4_useful} imply that
  $g(t) = 0$ unless
  \begin{equation}\label{eq:A3+A1_second_fourth_height_cond_ideals}
    \N\eI_1^2\N\eI_2^2\N\eI_3^3\N\eI_4^2\N\eI_6 \leq B\quad \text{ and }\quad t \leq t_2 := \left(\frac{4B}{\N\eI_3\N\eI_4^2\N\eI_5^4\N\eI_6^3}\right)^{1/2}\text.
  \end{equation}
  Moreover, applying \cite[Lemma 3.4, (2)]{arXiv:1302.6151} to
  \eqref{eq:A3+A1_height_5}, we see that
  \begin{align*}
    g(t) &\ll \frac{1}{\N\eI_1}\cdot \left(\frac{\N\eI_1 B}{\N\eI_2
        t}\right)^{1/2}\\ &= \frac{B}{\N\eI_1 \cdots \N\eI_6
      t}\left(\frac{B}{\N\eI_1^2\N\eI_2^2\N\eI_3^3\N\eI_4^2\N\eI_6}\right)^{-1/4}\left(\frac{B}{\N\eI_3\N\eI_4^2\N\eI_5^4\N\eI_6^3t^2}\right)^{-1/4}\text.
  \end{align*}
  In particular, we always have $g(t) \ll B/(\N\eI_1 \cdots \N\eI_6 t)$.

  By \cite[Lemma 5.4, Lemma 2.2]{arXiv:1302.6151}, $\vartheta$ satisfies the
  condition \cite[(6.1)]{arXiv:1302.6151} with $C = 0$ and $c_\vartheta =
  2^{\omega(\eI_1 \cdots \eI_4 \eI_6)}$.
 
  Let $t_1 := (\log B)^{14}$. A straightforward application of \cite[Proposition
  6.1]{arXiv:1302.6151} would not yield sufficiently good error terms, so,
  using a strategy as in the proof of \cite[Proposition
  7.2]{arXiv:1302.6151}, we split the sum over $\e_7$ into the two cases
  $\N\eI_7 \leq t_1$ and $\N\eI_7 > t_1$.

  Let us start with the second case. We may assume that $t_2 \geq t_1$. Using
  \cite[Proposition 6.1]{arXiv:1302.6151} with the upper bound
  $g(t) \ll B/(\N\eI_1\cdots \N\eI_6 t$), we see that
  \begin{align*}
    \sum_{\substack{\e_7 \in \OO_{7*}\\\N\eI_7 >
        t_1}}\vartheta(\eI_7)g(\N\eI_7) &= \frac{2
      \pi}{\sqrt{|\Delta_K|}}\mathcal{A}(\vartheta(\aaa),\aaa, \OO_K)\int\limits_{t
      \geq t_1}g(t)\dd t\\ &+ O\left(\frac{2^{\omega_K(\eI_1\cdots
          \eI_4\eI_6)}B}{\N\eI_1 \cdots \N\eI_6}t_1^{-1/2}\right)\text.
  \end{align*}
  When summing the error term over the remaining variables, we may sum
  over all $\eII''$ with $\N\eI_j \leq B$, so the error term is
  \begin{equation*}
      \ll t_1^{-1/2}\sum_{\eII''}\frac{2^{\omega_K(\eI_1\cdots
          \eI_4\eI_6)}B}{\N\eI_1\cdots\N\eI_6} \ll (\log B)^{-7}B(\log
      B)^{11} = B(\log B)^4\text.
  \end{equation*}
  Now let us consider the sum over all $\e_7$ with $\N\eI_7 \leq
  t_1$. Since $0\leq\vartheta(\eI_7) \leq 1$, we obtain an upper bound
  \begin{align*}
    &\ \sum_{\ee'' \in \OO_{1*} \times \dots \times \OO_{6*}}\sum_{\substack{\e_7 \in \OO_{7*}\\\N\eI_7 \leq t_1}}\vartheta(\eI_7)g(\N\eI_7)\\
    &\ll\hspace{-0.5cm}
    \sum_{\substack{\eII'',\eI_7\\\eqref{eq:A3+A1_second_fourth_height_cond_ideals}\text{
        with }t=\N\eI_7\\\N\eI_7 \leq t_1}}\hspace{-0.6cm}\frac{B}{\N\eI_1 \cdots \N\eI_6 \N\eI_7}\left(\frac{B}{\N\eI_1^2\N\eI_2^2\N\eI_3^3\N\eI_4^2\N\eI_6}\right)^{-\frac{1}{4}}\left(\frac{B}{\N\eI_3\N\eI_4^2\N\eI_5^4\N\eI_6^3\N\eI_7^2}\right)^{-\frac{1}{4}}\\
    &\ll\hspace{-0.5cm} \sum_{\substack{\eI_2, \ldots, \eI_7\\\eqref{eq:A3+A1_second_fourth_height_cond_ideals}\text{
        with }t=\N\eI_7\\\N\eI_7 \leq t_1}}\frac{B}{\N\eI_2 \cdots \N\eI_7}\left(\frac{B}{\N\eI_3\N\eI_4^2\N\eI_5^4\N\eI_6^3\N\eI_7^2}\right)^{-\frac{1}{4}}\\
    &\ll\hspace{-0.5cm} \sum_{\substack{\eI_2, \eI_3, \eI_4, \eI_6, \eI_7\\N\eI_j \leq B,\ \N\eI_7 \leq t_1}}\frac{B}{\N\eI_2 \N\eI_3 \N\eI_4 \N\eI_6 \N\eI_7} \ll B(\log B)^4\log t_1 \ll B(\log B)^4\log\log B\text.
  \end{align*}
  Our proof is finished once we see that
  \begin{equation*}
    \sum_{\ee'' \in \OO_{1*} \times \dots \times \OO_{6*}} \mathcal{A}(\vartheta(\aaa), \aaa) \int_{1}^{t_1}g(t)\dd t \ll B(\log B)^4\log\log B\text.
  \end{equation*}
  This follows from an analogous computation as above with the
  integral over $t$ instead of the sum over $\eI_7$, and using that $0
  \leq \mathcal{A}(\vartheta(\aaa), \aaa) \leq 1$.
\end{proof}

\begin{lemma}\label{lem:A3+A1_second_summation_ideals}
  If $\eII''$ runs over all six-tuples $(\eI_1, \ldots, \eI_6)$ of
  nonzero ideals of $\OO_K$ then we have
  \begin{align*}
    N_{U_1,H}(B) &=
    \left(\frac{2}{\sqrt{|\Delta_K|}}\right)^2\sum_{\eII''}\mathcal{A}(\theta_8(\eII'',
    \eI_7), \eI_7) V_7(\N\eI_1, \ldots, \N\eI_6; B)\\ &+ O(B(\log B)^4\log
    \log B)\text.
  \end{align*}
\end{lemma}

\begin{proof}
  This is entirely analogous to \cite[Lemma 9.4]{arXiv:1302.6151}.
\end{proof}

\subsubsection{The remaining summations}
\begin{lemma}\label{lem:A3+A1_completion}
  We have
  \begin{equation*}
    N_{U_1,H}(B) = \left(\frac{2}{\sqrt{|\Delta_K|}}\right)^8 \left(\frac{h_K}{\omega_K}\right)^6 \theta_0V_0(B) + O(B(\log B)^4\log \log B),
  \end{equation*}
  where $\theta_0$ is as in \eqref{eq:def_theta_0} and
  \begin{equation*}
    V_0(B) := \int\limits_{\substack{(\e_1, \ldots, \e_8)\in\mathcal{R}(B)\\\abs{\e_1}, \dots, \abs{\e_7} \ge 1}}\frac{1}{\abs{\e_1}}\dd \e_1 \cdots \dd \e_8,
  \end{equation*}
  with complex variables $\e_1, \ldots, \e_8$.
\end{lemma}

\begin{proof}
  By \cite[Lemma 3.4, (6)]{arXiv:1302.6151}, applied to the
  \eqref{eq:A3+A1_height_5}, we have
  \begin{equation*}
    V_7(t_1, \ldots, t_6; B) \ll \frac{B^{2/3}}{t_1^{1/3}t_2^{1/3}t_4^{1/3}t_5t_6^{2/3}}
    = \frac{B}{t_1 \cdots
      t_6}\left(\frac{B}{t_1^2t_2^2t_3^3t_4^2t_6}\right)^{-1/3}\text.
  \end{equation*}
  We apply \cite[Proposition 7.3]{arXiv:1302.6151} with $r=5$ and use polar
  coordinates, similarly to \cite[Lemma 9.5, Lemma 9.9]{arXiv:1302.6151}.
\end{proof}

\subsection{Proof of Theorem \ref{thm:main} for $S_1$}
We will use the conditions
\begin{align}
  &\abs{\e_1^2\e_2^2\e_4^2\e_6}\leq B\text{ and }\label{eq:A3+A1_comparison_1}\\
  &\abs{\e_1^2\e_2^2\e_4^2\e_6}\leq B \text{ and
  }\abs{\e_1^{-1}\e_2^{-1}\e_4^2\e_5^{6}\e_6^{4}}\le
  B.\label{eq:A3+A1_comparison_2}
\end{align}
\begin{lemma}\label{lem:A3+A1_predicted_volume}
  Let $\alpha(\tS_1)$, $\omega_\infty(\tS_1)$ be as in Theorem \ref{thm:main},
  let $\mathcal{R}(B)$ be as in
  \eqref{eq:A3+A1_height_1}--\eqref{eq:A3+A1_height_5}, and define
    \begin{equation*}
    V_0'(B) := \int_{\substack{(\e_1, \ldots, \e_8) \in
        \mathcal{R}(B)\\\abs{\e_1}, \abs{\e_2}, \abs{\e_4}, \abs{\e_5}, \abs{\e_6}\geq 1\\\eqref{eq:A3+A1_comparison_2}}} \frac{1}{\abs{\e_1}}
    \dd \e_1 \cdots \dd \e_8,
  \end{equation*}
  with complex variables $\e_1$, $\ldots$, $\e_8$. Then
  \begin{equation}\label{eq:A3+A1_predicted_volume}
    \pi^6\alpha(\tS_1) \omega_\infty(\tS_1) B(\log B)^5 = 4 V_0'(B).
  \end{equation}
\end{lemma}

\begin{proof}
  We use the following substitutions on $\omega_\infty(\tS_1)$: Let $\e_1$,
  $\e_2$, $\e_4$, $\e_5$, $\e_6 \in \CC\smallsetminus\{0\}$ and $B > 0$. Let
  $\e_3$, $\e_7$, $\e_8$ be complex variables. With $l :=
  (B\abs{\e_1\e_2\e_4\e_5^3\e_6^2})^{1/2}$, we apply the coordinate
  transformation $z_0 = l^{-1/3}\e_2\cdot \e_8$, $z_1 =
  l^{-1/3}\e_1\e_2\e_4\e_5\e_6 \cdot \e_3$, $z_2 =
  l^{-1/3}(-\e_2\cdot\e_8 - \e_4\e_5^3\e_6^2\cdot\e_7)$, of Jacobi
  determinant
  \begin{equation}
    \frac{\abs{\e_1\e_2\e_4\e_5\e_6}}{B}\frac{1}{\abs{\e_1}},
  \end{equation}
  and obtain
  \begin{equation}\label{eq:A3+A1_complex_density_torsor}
    \omega_\infty(\tS_1) = \frac{12}{\pi}\frac{\abs{\e_1\e_2\e_4\e_5\e_6}}{B}
    \int_{(\e_1, \ldots, \e_8) \in \mathcal{R}(B)}\frac{1}{\abs{\e_1}}\dd \e_3 \dd \e_7 \dd \e_8\text.
  \end{equation}

  The negative curves $[E_1], \dots, [E_7]$ generate the effective cone of
  $\tS_1$. We have $[-K_{\tS_1}] = [2E_1+2E_2+3E_3+2E_4+E_6]$ and $[E_7]
  =[E_1+E_2+E_3-2E_5-E_6]$. Hence, \cite[Lemma 8.1]{arXiv:1302.6151} (with the
  roles of $\e_3$ and $\e_6$ exchanged) gives
  \begin{equation}\label{eq:A3+A1_alpha}
     \alpha(\tS_1)(\log B)^5=\frac 1{3 \pi^5}
      \int\limits_{\substack{\abs{\e_1}, \abs{\e_2}, \abs{\e_4}, \abs{\e_5},
          \abs{\e_6} \geq 1\\\eqref{eq:A3+A1_comparison_2}}}
      \frac{\dd \e_1\dd \e_2\dd \e_4\dd \e_5\dd
        \e_6}{\abs{\e_1\e_2\e_4\e_5\e_6}}\text.
  \end{equation}
  The lemma follows by substituting
  \eqref{eq:A3+A1_complex_density_torsor} and \eqref{eq:A3+A1_alpha}
  in \eqref{eq:A3+A1_predicted_volume}.
\end{proof}

To finish our proof, we compare $V_0(B)$ from Lemma
\ref{lem:A3+A1_completion} with $V_0'(B)$ defined in Lemma
\ref{lem:A3+A1_predicted_volume}. Let 
\begin{align*}
  \mathcal{D}_0(B) &:= \{(\e_1, \ldots, \e_8) \in \mathcal{R}(B) \mid \abs{\e_1},\ldots,\abs{\e_7} \geq 1\},\\
  \mathcal{D}_1(B) &:= \{(\e_1, \ldots, \e_8) \in \mathcal{R}(B) \mid \abs{\e_1},\ldots,\abs{\e_7} \geq 1\text{, }\eqref{eq:A3+A1_comparison_1}\},\\
  \mathcal{D}_2(B) &:= \{(\e_1, \ldots, \e_8) \in \mathcal{R}(B) \mid \abs{\e_1},\ldots,\abs{\e_7} \geq 1\text{, }\eqref{eq:A3+A1_comparison_2}\},\\
  \mathcal{D}_3(B) &:= \{(\e_1, \ldots, \e_8) \in \mathcal{R}(B) \mid \abs{\e_1},\ldots,\abs{\e_6} \geq 1\text{, }\eqref{eq:A3+A1_comparison_2}\},\\
  \mathcal{D}_4(B) &:= \{(\e_1, \ldots, \e_8) \in \mathcal{R}(B) \mid
  \abs{\e_1},\!\abs{\e_2},\!\abs{\e_4},\!\abs{\e_5},\!\abs{\e_6} \geq 1\text{,
    }\eqref{eq:A3+A1_comparison_2}\}.
\end{align*}
Moreover, let
\begin{equation*}
  V_i(B) := \int_{\mathcal{D}_i(B)}\frac{\dd \e_1 \cdots \dd \e_8}{\abs{\e_1}}\text.
\end{equation*}
Then $V_0(B)$ is as in Lemma \ref{lem:A3+A1_completion} and
$V_4(B) = V_0'(B)$. We show that, for $1 \leq i \leq 4$, $V_i(B) -
V_{i-1}(B) = O(B(\log B)^4)$. This holds for $i = 1$, since, by
\eqref{eq:A3+A1_height_2} and $\abs{\e_3} \geq 1$, we have $\mathcal{D}_1(B) =
\mathcal{D}_0(B)$.

Moreover, using \cite[Lemma 3.4, (2)]{arXiv:1302.6151} and
(\ref{eq:A3+A1_height_5}) to bound the integral over $\e_8$, we have
\begin{equation*}
  V_2(B) - V_1(B) \ll \int_{\substack{1 \leq \abs{\e_1}, \ldots, \abs{\e_7}
      \leq B\\\abs{\e_1^{-1}\e_2^{-1}\e_4^2\e_5^{6}\e_6^{4}} > B\\ \eqref{eq:A3+A1_height_4_useful}}}\frac{B^{1/2}}{\abs{\e_1\e_2\e_7}^{1/2}}\dd \e_1 \cdots \dd \e_7 \ll B(\log B)^4\text.
\end{equation*}
Moreover,
\begin{equation*}
  V_3(B) - V_2(B) \ll \int\limits_{\substack{\abs{\e_1},\ldots,\abs{\e_6}\geq 1\\\abs{\e_7} <
      1\text{, }\eqref{eq:A3+A1_height_2}\text{, }\eqref{eq:A3+A1_comparison_2}}}\frac{B^{1/2}}{\abs{\e_1\e_2\e_7}^{1/2}}\dd \e_1 \cdots \dd \e_7 \ll B(\log B)^4\text.  
\end{equation*}

Finally, using \cite[Lemma 3.4, (4)]{arXiv:1302.6151} and (\ref{eq:A3+A1_height_5}) to bound the
integral over $\e_7$, $\e_8$, we have
\begin{equation*}
  V_4(B) - V_3(B) \ll\hspace{-1.3cm} \int\limits_{\substack{\abs{\e_1},\abs{\e_2},\abs{\e_4},\abs{\e_5},\abs{\e_6}\geq
      1\\\abs{\e_3} < 1\text{, }\eqref{eq:A3+A1_comparison_1}}}\frac{B^{2/3}}{\abs{\e_1\e_2\e_4\e_5^3\e_6^2}^{1/3}}\dd \e_1 \cdots \dd \e_6 \ll B(\log B)^4\text.
\end{equation*}
Using Lemma~\ref{lem:A3+A1_completion} and
Lemma~\ref{lem:A3+A1_predicted_volume}, this shows Theorem \ref{thm:main} for $S_1$.

\section{The quartic del Pezzo surface of type $\Afour$}

\subsection{Passage to a universal torsor}
We use the notation of \cite{math.AG/0604194}, except that we swap
$\eta_8$ and $\eta_9$.

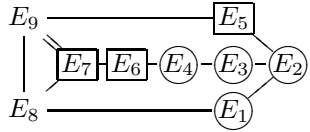
\begin{figure}[ht]
  \centering
  \[\xymatrix@R=0.05in @C=0.05in{E_9 \ar@{-}[rrrr] \ar@{-}[dd] \ar@{=}[dr] & & & & \li{5} \ar@{-}[dr]\\
    & \li{7} \ar@{-}[r] & \li{6} \ar@{-}[r] & \ex{4} \ar@{-}[r] & \ex{3} \ar@{-}[r] & \ex{2} \ar@{-}[dl]\\
    E_8 \ar@{-}[rrrr] \ar@{-}[ur] & & & & \ex{1}}\]
  \caption{Configuration of curves on $\tS_2$}
  \label{fig:A4_dynkin}
\end{figure}

For any given $\classtuple = (C_0, \dots, C_5) \in \classrep^6$, we
define $u_\classtuple := \N(C_0^3 C_1^{-1}\cdots C_5^{-1})$ and
\begin{align*}
    \OO_1 &:= C_3C_4^{-1} & \OO_2 &:= C_4C_5^{-1} & \OO_3 &:= C_0C_1^{-1}C_3^{-1}C_4^{-1}\\
    \OO_4 &:= C_1C_2^{-1} & \OO_5 &:= C_5 & \OO_6 &:= C_2 \\
    \OO_7 &:= C_0 C_1^{-1} C_2^{-1} & \OO_8 &:= C_0 C_3^{-1} & \OO_9 &:=
    C_0^2 C_3^{-1} C_4^{-1} C_5^{-1}.
\end{align*}
Let
\begin{equation*}
  \OO_{j*} :=
  \begin{cases}
    \OO_j^{\neq 0}, & j \in \{1,\ldots, 7\},\\
    \OO_j, & j \in \{8,9\}.
  \end{cases}
\end{equation*}
For $\eta_j \in \OO_j$, let
\begin{equation*}
  \eI_j := \e_j \OO_j^{-1}\text.
\end{equation*}
For $B \geq 0$, let $\mathcal{R}(B)$ be the set of all $(\e_1, \ldots, \e_8)
\in \CC^8$ with $\e_5 \neq 0$ and
\begin{align}
  \abs{\e_1^2\e_2^4\e_3^3\e_4^2\e_5^3\e_6}&\leq B,\label{eq:A4_height_1}\\
  \abs{\e_1\e_2\e_3\e_4\e_6\e_7\e_8}&\leq B,\label{eq:A4_height_2}\\
  \abs{\e_1^2\e_2^3\e_3^2\e_4\e_5^2\e_8}&\leq B,\label{eq:A4_height_3}\\
  \abs{\e_1\e_2^2\e_3^2\e_4^2\e_5\e_6^2\e_7}&\leq B,\label{eq:A4_height_4}\\
  \abs{\frac{\e_1\e_7\e_8^2 + \e_3\e_4^2\e_6^3\e_7^2}{\e_5}}&\leq B,\label{eq:A4_height_5}
\end{align}
and let $M_\classtuple(B)$ be the set of all
\begin{equation*}
  (\e_1, \ldots, \e_9) \in \OO_{1*} \times \cdots \times \OO_{9*} 
\end{equation*}
that satisfy the \emph{height conditions}
\begin{equation*}\label{eq:A4_height}
  (\e_1, \ldots, \e_8) \in \mathcal{R}(u_\classtuple B)\text,
\end{equation*}
the \emph{torsor equation}
\begin{equation}\label{eq:A4_torsor}
  \e_3\e_4^2\e_6^3\e_7 + \e_1\e_8^2 + \e_5\e_9 = 0,
\end{equation}
and the \emph{coprimality conditions}
\begin{equation}\label{eq:A4_coprimality}
  \eI_j + \eI_k = \OO_K \text{ for all distinct nonadjacent vertices $E_j$, $E_k$ in Figure~\ref{fig:A4_dynkin}.}
\end{equation}

\begin{lemma}\label{lem:A4_passage_to_torsor} We have
  \begin{equation*}
    N_{U_2,H}(B) = \frac{1}{\omega_K^6}\sum_{\classtuple \in \classrep^6}|M_\classtuple(B)|\text.
  \end{equation*}
\end{lemma}

\begin{proof}
  This is a specialization of \cite[Claim 4.1]{arXiv:1302.6151} and we
  prove it using the strategy from \cite[Section~4]{arXiv:1302.6151} with the
  data supplied in \cite{math.AG/0604194}. Starting
  with the curves $E_3^{(0)}:= \{y_0=0\}$, $E_8^{(0)} := \{y_1=0\}$, $E_7^{(0)} := \{y_2 = 0\}$, $E_9^{(0)} :=
  \{-y_0y_2 - y_1^2 = 0\}$ in  $\PP^2_K$, we prove \cite[Claim
  4.2]{arXiv:1302.6151} for the following sequence of blow-ups:
  \begin{enumerate}
  \item blow up $E_3^{(0)}\cap E_8^{(0)} \cap E_9^{(0)}$, giving $E_1^{(1)}$,
  \item blow up $E_1^{(1)}\cap E_3^{(1)} \cap E_9^{(1)}$, giving $E_2^{(2)}$,
  \item blow up $E_2^{(2)} \cap E_9^{(2)}$, giving $E_5^{(3)}$,
  \item blow up $E_3^{(3)} \cap E_7^{(3)}$, giving $E_4^{(4)}$,
  \item blow up $E_4^{(4)} \cap E_7^{(4)}$, giving $E_6^{(5)}$.
  \end{enumerate}

  The inverse $\pi \circ \rho^{-1} : \PP^2_K \rto S_2$ of the
  projection $\phi = \rho \circ \pi^{-1} : S_2 \rto \PP^2_K$, $(x_0 : \cdots :
  x_4) \mapsto (x_0 : x_2 : x_3)$ given by
  \begin{equation}\label{eq:A4_parameterization}
    (y_0 : y_1 : y_2) \mapsto (y_0^3 : y_0y_1y_2 : y_0^2y_1 : y_0^2y_2 : -y_2(y_1^2 + y_0y_2)),
  \end{equation}  
  and the map $\Psi$ appearing in \cite[Claim 4.2]{arXiv:1302.6151} sends $(\e_1,
  \ldots, \e_9)$ to
  \begin{equation*}
   (\e_1^2\e_2^4\e_3^3\e_4^2\e_5^3\e_6, \e_1\e_2\e_3\e_4\e_6\e_7\e_8, \e_1^2\e_2^3\e_3^2\e_4\e_5^2\e_8, \e_1\e_2^2\e_3^2\e_4^2\e_5\e_6^2\e_7, \e_7\e_9).
  \end{equation*}
  As in the proof of \cite[Lemma 9.1]{arXiv:1302.6151}, we see that the
  hypotheses of \cite[Lemma 4.3]{arXiv:1302.6151} are satisfied, so \cite[Claim
  4.2]{arXiv:1302.6151} holds in our situation for $i=0$.

  Note that \cite[Lemma~4.4]{arXiv:1302.6151} applies in steps (3), (4), (5) of
  the above chain of blow-ups. In steps (1), (2), we are in the situation of
  \cite[Remark~4.5]{arXiv:1302.6151}, so that we must derive some coprimality
  conditions using the torsor equation. We use the notation of \cite[Lemma~4.4,
  Remark~4.5]{arXiv:1302.6151}.

  For (1), we start with the parameterization provided by
  \cite[Lemma~4.3]{arXiv:1302.6151}, consisting of $(\e_3',\e_7',\e_8',\e_9')$
  satisfying certain coprimality conditions and other conditions. Since $\e_3'
  \ne 0$, there is a unique $C_1 \in \classrep$ such that
  $[I_3'+I_8'+I_9']=[C_1^{-1}]$. We choose $\e_1'' \in C_1$ such that
  $I_1''=I_3'+I_8'+I_9'$; this is unique up to multiplication by
  $\OO_K^\times$. We define $\e_3'':=\e_3'/\e_1''$, $\e_8'':=\e_8'/\e_1''$,
  $\e_9'':=\e_9'/\e_1''$ and $\e_7'':=\e_7'$. To show that $(\e_1'', \e_3'',
  \e_7'', \e_8'', \e_9'')$ lies in the set described in
  \cite[Claim~4.2]{arXiv:1302.6151} for $i=1$, everything is provided by the
  proof of \cite[Lemma~4.4]{arXiv:1302.6151} except the coprimality conditions
  involving $\e_1'', \e_3'', \e_8'', \e_9''$. Considering the configuration of
  $E_1^{(1)}, E_3^{(1)}, E_8^{(1)}, E_9^{(1)}$, these are
  $\eI_3''+\eI_8''=\OO_K$ (which holds because $I_3''+I_8''+I_9''=\OO_K$ by
  construction and because of the relation
  $\e_3''\e_7''+\e_1''\e_8''^2+\e_9''=0$) and $\eI_1''+\eI_8''+\eI_9''=\OO_K$
  (which holds because otherwise the relation would give non-triviality of
  $\eI_1''+\eI_8''+\eI_9''+\eI_3''\eI_7''$ contradicting the previous
  condition or the condition $\eI_1''+\eI_7''=\OO_K$ provided by the proof of
  \cite[Lemma~4.4]{arXiv:1302.6151}).

  For (2), we replace $''$ by $'$ in the result of the previous step. We
  choose $C_2 \in \classrep$ such that $[I_1'+I_3'+I_9']=[C_2^{-1}]$ and
  $\e_2'' \in C_4 = \OO_2''$ such that $I_2''=I_1'+I_3'+I_9'$. It remains to
  check the pairwise coprimality of $I_1'', I_3'', I_9''$. By construction,
  $I_1''+I_3''+I_9''=\OO_K$; considering the torsor equation
  $\e_3''\e_7''+\e_1''\e_8''^2+\e_9''=0$ shows $I_1''+I_3''=\OO_K$ directly,
  $I_1''+I_9''=\OO_K$ using $I_1''+I_7''=\OO_K$, and $I_3''+I_9''=\OO_K$ using
  $I_3''+I_8''=\OO_K$.

  Since steps (3), (4), (5) are covered by \cite[Lemma 4.4]{arXiv:1302.6151},
  this shows \cite[Claim 4.2]{arXiv:1302.6151}. We deduce \cite[Claim
  4.1]{arXiv:1302.6151} in the same way as in \cite[Lemma 9.1]{arXiv:1302.6151}.
\end{proof}

\subsection{Summations}
\subsubsection{The first summation over $\e_8$ with dependent $\e_9$}
Let $\ee' := (\e_1, \ldots, \e_7)$ and $\eII' := (\eI_1, \ldots, \eI_7)$. Let
$\theta_0(\eII') := \prod_\p\theta_{0,\p}(J_\p(\eII'))$, where $J_\p(\eII') :=
\{j \in \{1, \ldots, 7\}\ :\  \p \mid \eI_j\}$ and
\begin{equation*}
  \theta_{0,\p}(J):=
  \begin{cases}
    1 &\text{ if } J = \emptyset, \{1\}, \{2\}, \{3\}, \{4\}, \{5\}, \{6\},
    \{7\},\\
    \ &\text{ or } J= \{1,2\}, \{2,3\}, \{2,5\}, \{3,4\}, \{4,6\}, \{6,7\},\\
    0 &\text{ otherwise.} 
  \end{cases}
\end{equation*}
Then $\theta_{0}(\eII') = 1$ if and only if $\eI_1$, $\ldots$, $\eI_7$
satisfy the coprimality conditions from \eqref{eq:A4_coprimality}, and
$\theta_{0}(\eII') = 0$ otherwise.

We apply \cite[Proposition 5.3]{arXiv:1302.6151} with $(A_1, A_2, A_3, A_0) :=
(3,4,6,7)$, $(B_1, B_0) := (1, 8)$, $(C_1, C_0):=(5,9)$, and $D := 2$. For
given $\e_2$, $\e_5$, we write
\begin{equation*}
  \e_3\e_4^2\e_6^3\e_7 = \e_{A_0}^{a_0}\Pi(\eeA) = \Ao \At^2,
\end{equation*}
where $\Ao$, $\At$ are chosen as follows: Let $\id A = \id A(\e_2, \e_5)$ be a prime ideal not
dividing $\eI_2\eI_5$ such that $\id A\OO_6^{-1} \OO_8 = \id A C_0C_2^{-1}C_3^{-1}$
is a principal fractional ideal $t\OO_K$, for a suitable $t = t(\e_2, \e_5) \in
K^\times$. Then we define $\At = \At(\e_2,\e_5) := \e_6t$ and $\Ao := \Ao(\e_2, \e_5)
:= \e_3\e_4^2\e_6\e_7t^{-2}$.

\begin{lemma}\label{lem:A4_first_summation}
  We have
  \begin{equation*}
    |M_\classtuple(B)| = \frac{2}{\sqrt{|\Delta_K|}}\sum_{\ee' \in \OO_{1*} \times \dots \times \OO_{7*}} \theta_8(\ee', \classtuple)V_8(\N\eI_1,\ldots,\N\eI_7; B) + O_\classtuple(B(\log B)^3),
  \end{equation*}
  where
  \begin{equation*}
    V_8(t_1, \ldots, t_7; B) :=  \frac{1}{t_5}\int_{(\sqrt{t_1}, \ldots,
      \sqrt{t_7},\e_8) \in \mathcal{R}(B)} \dd \e_8.
  \end{equation*}
  Moreover,
  \begin{equation*}
    \theta_8(\ee', \classtuple) := \sum_{\substack{\kc \mid \eI_2\\\kc + 
        \eI_1\eI_3 =
        \OO_K}}\frac{\mu_K(\kc)}{\N\kc}\tilde\theta_8(\eII',\kc)
    \sum_{\substack{\rho \mod \kc\eI_5\\\rho\OO_K + \kc\eI_5 = \OO_K\\\rho^2 \equiv_{\kc\eI_5} 
        \e_6\e_7 A}}1\text,
  \end{equation*}
  with
  \begin{equation*}
  \tilde\theta_8(\eII',\kc) := \theta_0(\eII')\frac{\phi_K^*(\eI_2\eI_3\eI_4\eI_6)}{\phi_K^*(\eI_2+
      \kc\eI_5)}\text.
  \end{equation*}
  Here, $A := -\e_3\e_4^2/(t(\e_2,\e_5)^2\e_1)$, and $\e_6\e_7 A$ is invertible
  modulo $\kc\eI_5$ whenever $\theta_0(\eII')\neq 0$.
\end{lemma}

\begin{proof}
  It is clear that $\theta_8(\ee', \classtuple) = \theta_1(\ee')$ from \cite[Proposition
  5.3]{arXiv:1302.6151}, and a simple argument as in the proof of \cite[Lemma
  9.2]{arXiv:1302.6151} shows that $V_8(\N\eI_1, \ldots, \N\eI_7; B) =
  V_1(\ee', u_\classtuple B)$. Hence, the main term is correct and it remains to
  bound the error term arising from \cite[Proposition
  5.3]{arXiv:1302.6151}.

  Similarly as in \cite[Lemma 9.2]{arXiv:1302.6151}, we see that the set
  $\mathcal{R}(\ee', B)$ of all $\e_8$ with $(\e_1, \ldots, \e_8) \in
  \mathcal{R}(u_\classtuple B)$ is of bounded class and (using \cite[Lemma 3.5,
  (1)]{arXiv:1302.6151} on \eqref{eq:A4_height_5}) contained in two balls of
  radius $R(\ee'; u_\classtuple B) \ll_\classtuple
  \left(B\N\eI_5\N\eI_1^{-1}\N\eI_7^{-1}\right)^{1/4}$.
  
  The error term is
  \begin{equation*}
    \ll \sum_{\ee', \eqref{eq:A4_fourth_height_cond_ideals}}2^{\omega_K(\eI_2)+\omega_K(\eI_2\eI_3\eI_4\eI_6)+\omega_K(\eI_2\eI_5)}\left(\frac{R(\ee'; u_\classtuple B)}{\N(\eI_5)^{1/2}}+1\right)\text,
  \end{equation*}
  where, using \eqref{eq:A4_height_4}, the sum runs over all $\ee' \in \OO_{1*}
  \times \cdots \times \OO_{7*}$ with
  \begin{equation}\label{eq:A4_fourth_height_cond_ideals}
    \N(\eI_1\eI_2^2\eI_3^2\eI_4^2\eI_5\eI_6^2\eI_7) \leq B\text.
  \end{equation}
  Since $|\OO_K^\times| < \infty$, we can sum over the $I_j$ instead of the
  $\e_j$, which then run over all nonzero ideals of $\OO_K$ with
  (\ref{eq:A4_fourth_height_cond_ideals}), and obtain
  \begin{align*}
    &\ll_\classtuple \sum_{\eII'\text{, }\eqref{eq:A4_fourth_height_cond_ideals}}2^{\omega_K(\eI_2)+\omega_K(\eI_2\eI_3\eI_4\eI_6)+\omega_K(\eI_2\eI_5)}\left(\frac{B^{1/4}}{\N\eI_1^{1/4}\N\eI_5^{1/4}\N\eI_7^{1/4}}+1\right)\\
    &\ll\sum_{\eI_1, \dots, \eI_6}\left(\frac{2^{\omega_K(\eI_2)+\omega_K(\eI_2\eI_3\eI_4\eI_6)+\omega_K(\eI_2\eI_5)}B}{\N\eI_1\N\eI_2^{3/2}\N\eI_3^{3/2}\N\eI_4^{3/2}\N\eI_5\N\eI_6^{3/2}}+\frac{2^{\omega_K(\eI_2)+\omega_K(\eI_2\eI_3\eI_4\eI_6)+\omega_K(\eI_2\eI_5)}B}{\N\eI_1\N\eI_2^2\N\eI_3^2\N\eI_4^2\N\eI_5\N\eI_6^2}\right)\\
    &\ll B(\log B)^3 + B(\log B)^3 \ll B(\log B)^3.\qedhere
  \end{align*}
\end{proof}

For the further summations, we define
\begin{equation*}
  \theta_8'(\eII') := \sum_{\substack{\kc \mid \eI_2\\\kc + 
      \eI_1\eI_3 =
      \OO_K}}\frac{\mu_K(\kc)}{\N\kc}\tilde\theta_8(\eII',\kc)
\end{equation*}
and distinguish between two cases:
Similarly to \cite{MR2543667}, let $M^{(86)}_\classtuple(B)$ be the main
term in Lemma \ref{lem:A4_first_summation} with the additional
condition $\N\eI_6 > \N\eI_7$ on the $\ee'$, and let
$M^{(87)}_\classtuple(B)$ be the main term with the additional condition
$\N\eI_6 \leq \N\eI_7$.
Moreover, we define
\begin{equation*}
  N_{86}(B) :=\frac{1}{\omega_K^{6}}\sum_{\classtuple \in \classrep^6}M^{(86)}_\classtuple(B)
\end{equation*}
and $N_{87}(B)$ analogously, so
\begin{equation}\label{eq:A4_split_summation}
  N_{U_2, H}(B) = N_{86}(B) + N_{87}(B) + O(B(\log B)^3)\text.
\end{equation}

\subsubsection{The second summation over $\e_6$ in $M^{(86)}_\classtuple(B)$}
\begin{lemma}\label{lem:A4_second_summation_Ma}
  Write $\ee'' := (\e_1, \dots, \e_5, \e_7)$ and $\OO'':=\OO_{1*}\times \cdots
  \times \OO_{5*}\times\OO_{7*}$. We have
  \begin{align*}
    M^{(86)}_\classtuple(B) &=\left(\frac{2}{\sqrt{|\Delta_K|}}\right)^2 \sum_{\ee''\in \OO''} \mathcal{A}(\theta_8'(\eII'),\eI_6)V_{86}(\N\eI_1,\ldots,\N\eI_5,\N\eI_7; B)\\ &+ O_\classtuple(B(\log B)^4),
  \end{align*}
  where, for $t_1, \ldots, t_5, t_7 \geq 1$,
  \begin{equation*}
    V_{86}(t_1, \ldots, t_5, t_7; B) := \frac{\pi}{t_5}\int\limits_{\substack{(\sqrt{t_1},
        \ldots, \sqrt{t_7}, \e_8) \in \mathcal{R}(B)\\t_6>t_7}} \dd t_6 \dd \e_8\text,
  \end{equation*}
  with a real variable $t_6$ and a complex variable $\e_8$.
\end{lemma}

\begin{proof}
  We follow the strategy described in \cite[Section 6]{arXiv:1302.6151} in the
  case $b_0 \geq 2$.  We write
  \begin{equation*}
    M^{(86)}_\classtuple(B) = \frac{2}{\sqrt{|\Delta_K|}}\sum_{\ee'' \in \OO''} \sum_{\substack{\kc \mid \eI_2\\\kc+\eI_1\eI_3 = \OO_K}}\frac{\mu(\kc)}{\N\kc}\Sigma,
  \end{equation*}
  where
  \begin{equation*}
    \Sigma := \sum_{\substack{\e_6 \in \OO_{6*}\\\N\eI_6 >
        \N\eI_7}}\vartheta(\eI_6)\sum_{\substack{\rho \mod
        \kc\eI_5\\\rho\OO_K + \kc\eI_5 = \OO_K \\\rho^2 \equiv_{\kc\eI_5} \e_6\e_7 A}}g(\N\eI_6)\text,
  \end{equation*}
  with $\vartheta(\eI_6) := \tilde\theta_8(\eII', \kc)$ and $g(t) :=
  V_8(\N\eI_1,\ldots,\N\eI_5, t, \N\eI_7; B)$.
 
  By \cite[Lemma 5.5, Lemma 2.2]{arXiv:1302.6151}, the function $\vartheta$
  satisfies \cite[(6.1)]{arXiv:1302.6151} with $C:=0$, $c_\vartheta :=
  2^{\omega_K(\eI_1\eI_2\eI_3\eI_5)}$. 
  By \eqref{eq:A4_height_4}, we have $g(t)=0$ whenever $t > t_2 :=
  B^{1/2}/(\N\eI_1^{1/2}\N\eI_2\N\eI_3\N\eI_4\N\eI_5^{1/2}\N\eI_7^{1/2})$, and,
  by Lemma \cite[Lemma 3.5, (2)]{arXiv:1302.6151} applied to
  \eqref{eq:A4_height_5}, we have $g(t) \ll B^{1/2}/(\N\eI_1^{1/2}\N\eI_5^{1/2}
  \N\eI_7^{1/2})$. Using \cite[Proposition 6.1]{arXiv:1302.6151}, we obtain
  \begin{align*}
    \Sigma &= \frac{2 \pi}{\sqrt{|\Delta_K|}}\phi_K^*(\kc\eI_5)\mathcal{A}(\vartheta(\aaa), \aaa, \kc\eI_5)\int_{t \geq \N\eI_7}g(t)\dd t\\
    &+ O\left(\frac{2^{\omega_K(\eI_1\eI_2\eI_3\eI_5)}B^{1/2}}{\N\eI_1^{1/2}\N\eI_5^{1/2}\N\eI_7^{1/2}}\left(\frac{B^{1/4}\N\kc^{1/2}\N\eI_5^{1/4}}{\N\eI_1^{1/4}\N\eI_2^{1/2}\N\eI_3^{1/2}\N\eI_4^{1/2}\N\eI_7^{1/4}}
        + \N(\kc\eI_5)\log B\right)\right)\text.
  \end{align*}
  Using \cite[Lemma 6.3]{arXiv:1302.6151} we see that the main term in the
  lemma is correct.

  For the error term, we may sum over $\kc$ and over the ideals $\eI_j$ instead of
  the $\e_j$, since $|\OO_K^\times| < \infty$. By \eqref{eq:A4_height_1} and our
  condition $\N\eI_6 > \N\eI_7$ it suffices to sum over $\kc$ and all $(\eI_1, \ldots,
  \eI_5, \eI_7)$ satisfying
  \begin{equation}\label{eq:A4_first_height_cond_I_7}
    \N\eI_1^2\N\eI_2^4\N\eI_3^3\N\eI_4^2\N\eI_5^3\N\eI_7 \leq B\text.
  \end{equation}
  Thus, the total error is bounded by
  \begin{align*}
    &\sum_{\substack{\eI_1, \ldots, \eI_5, \eI_7\\\eqref{eq:A4_first_height_cond_I_7}}}\left(\frac{2^{\omega_K(\eI_2) + \omega_K(\eI_1\eI_2\eI_3\eI_5)} B^{3/4}}{\N\eI_1^{3/4}\N\eI_2^{1/2}\N\eI_3^{1/2}\N\eI_4^{1/2}\N\eI_5^{1/4}\N\eI_7^{3/4}} + \frac{2^{\omega_K(\eI_2) + \omega_K(\eI_1\eI_2\eI_3\eI_5)}B^{1/2}\log B}{\N\eI_1^{1/2}\N\eI_5^{-1/2}\N\eI_7^{1/2}} \right)\\
    &\ll \sum_{\substack{\eI_1, \ldots, \eI_5\\\N\eI_j \leq B}}\left(\frac{2^{\omega_K(\eI_2) +\omega_K(\eI_1\eI_2\eI_3\eI_5)}B}{\N\eI_1^{5/4}\N\eI_2^{3/2}\N\eI_3^{5/4}\N\eI_4\N\eI_5} + \frac{2^{\omega_K(\eI_2) + \omega_K(\eI_1\eI_2\eI_3\eI_5)}B\log B}{\N\eI_1^{3/2}\N\eI_2^2\N\eI_3^{3/2}\N\eI_4\N\eI_5}\right)\\
    &\ll B(\log B)^3 + B(\log B)^4 \ll B(\log B)^4\text{.}\qedhere
  \end{align*}
\end{proof}

\begin{lemma}\label{lem:A4_second_summation_Ma_ideals}
  If $\eII''$ runs over all six-tuples $(\eI_1, \ldots, \eI_5, \eI_7)$
  of nonzero ideals of $\OO_K$ then we have
  \begin{equation*}
    N_{86}(B) = \left(\frac{2}{\sqrt{|\Delta_K|}}\right)^2\sum_{\eII''}\mathcal{A}(\theta_8'(\eII'),\eI_6)V_{86}(\N\eI_1, \ldots, \N\eI_5, \N\eI_7; B) + O(B(\log B)^4)\text. 
  \end{equation*}
\end{lemma}

\begin{proof}
 This is analogous to \cite[Lemma 9.4]{arXiv:1302.6151}.
\end{proof}

\subsubsection{The remaining summations for $N_{86}(B)$}
\begin{lemma}\label{lem:A4_completion_Ma}
  We have
  \begin{equation*}
    N_{86}(B) = \pi^6\left(\frac{2}{\sqrt{|\Delta_K|}}\right)^8 \left(\frac{h_K}{\omega_K}\right)^6 \theta_0V_{860}(B) + O(B(\log B)^4\log \log B),
  \end{equation*}
  where $\theta_0$ is as in \eqref{eq:def_theta_0} and
  \begin{equation*}
    V_{860}(B) := \int_{t_1, \dots, t_5, t_7 \ge 1}V_{86}(t_1, \dots, t_5, t_7; B) \dd t_1 \cdots \dd t_5 \dd t_7,
  \end{equation*}
  with real variables $t_1, \ldots, t_5, t_7$.
\end{lemma}

\begin{proof}
  By \cite[Lemma 3.5, (5)]{arXiv:1302.6151}, applied to \eqref{eq:A4_height_5}, we have, for $t_7 \geq 1$,
  \begin{equation*}
    V_{86}(t_1, \ldots, t_5, t_7; B) \ll \frac{B}{t_1 \cdots t_5
      t_7}\left(\frac{B}{t_1^3t_2^6t_3^4t_4^2t_5^5}\right)^{-1/6}
    \text.
  \end{equation*}
  Furthermore, using \eqref{eq:A4_height_1} to bound $t_6$ and \eqref{eq:A4_height_3} to bound $\abs{\e_8}$, we see that
  \begin{equation*}
    V_{86}(t_1, \ldots, t_5, t_7; B) \ll
    \frac{1}{t_5}\left(\frac{B}{t_1^2t_2^4t_3^3t_4^2t_5^3}\right)\left(\frac{B}{t_1^2t_2^3t_3^2t_4t_5^2}\right)
    =\frac{B}{t_1 \cdots t_5
      t_7}\left(\frac{B}{t_1^3t_2^6t_3^4t_4^2t_5^5}\right)\text.
  \end{equation*}
  We apply \cite[Proposition 7.3]{arXiv:1302.6151} with $r=5$.
\end{proof}

\subsubsection{The second summation over $\e_7$ in $M_\classtuple^{(87)}(B)$}
\begin{lemma}\label{lem:A4_second_summation_Mb}
  Write $\ee'' := (\e_1, \dots, \e_6)$. We have
  \begin{align*}
    M^{(87)}_\classtuple(B) &=\left(\frac{2}{\sqrt{|\Delta_K|}}\right)^2
    \sum_{\ee''\in \OO_{1*} \times \dots \times \OO_{6*}}
    \mathcal{A}(\theta_8'(\eII'), \eI_7)V_{87}(\N\eI_1,\ldots,\N\eI_6; B)\\ &+ O_\classtuple(B(\log B)^4),
  \end{align*}
  where, for $t_1,\ldots, t_6 \geq 1$,
  \begin{equation*}
    V_{87}(t_1, \ldots, t_6; B) := \frac{\pi}{t_5}\int\limits_{\substack{(\sqrt{t_1},
      \ldots, \sqrt{t_7}, \e_8)\in \mathcal{R}(B)\\t_7 \geq t_6}} \dd t_7 \dd \e_8\text,
  \end{equation*}
  with a real variable $t_7$ and a complex variable $\e_8$.
\end{lemma}

\begin{proof}
  Again, we apply the strategy described in \cite[Section 6]{arXiv:1302.6151}
  in the case $b_0 \geq 2$. However, this time we must examine the arithmetic
  function more carefully, since a straightforward application as in Lemma
  \ref{lem:A4_second_summation_Ma} would not yield sufficiently good error
  terms. We write
  \begin{equation}\label{eq:A4_sum_87_M87}
    M^{(87)}_\classtuple(B) = \frac{2}{\sqrt{|\Delta_K|}}\sum_{\ee'' \in
      \OO_{1*} \times \dots \times \OO_{6*}} \sum_{\substack{\kc \mid \eI_2\\\kc
        + \eI_1\eI_3 = \OO_K}}\frac{\mu_K(\kc)}{\N\kc}\Sigma,
  \end{equation}
  where
  \begin{equation}\label{eq:A4_sum_87_def_sigma}
    \Sigma := \sum_{\substack{\e_7 \in \OO_{7*}\\\N\eI_7 \geq
        \N\eI_6}}\vartheta(\eI_7)\sum_{\substack{\rho \mod \kc\eI_5\\\rho\OO_K +
        \kc\eI_5 = \OO_K \\\rho^2 \equiv_{\kc\eI_5} \e_6\e_7
        A}}g(\N\eI_7)\text,
  \end{equation}
  with $\vartheta(\eI_7) := \tilde\theta_8(\eII', \kc)$ and
  $g(t):=V_8(\N\eI_1, \ldots, \N\eI_6, t; B)$.
 
  The key observation is that, as in \cite{MR2543667}, we can replace $\vartheta(\eI_7)$ by the function
 \begin{equation*}
    \vartheta'(\eI_7) := \theta_0'(\eII')\frac{\phi_K^*(\eI_2\eI_3\eI_4\eI_6)}{\phi_K^*(\eI_2+
      \kc\eI_5)}\text,
  \end{equation*}
  where $\theta_0'$ encodes all coprimality conditions that are encoded by
  $\theta_0$, except for allowing $\eI_5 + \eI_7 \neq \OO_K$. For the
  representation $\theta_0' = \prod_\p \theta_{0,\p}'(J_\p(\eII'))$ as a
  product of local factors, this amounts to
  \begin{equation*}
    \theta'_{0,\p}(J):=
    \begin{cases}
      1 &\text{ if } \theta_{0,\p}(J) = 1 \text{ or } J=\{5,7\},\\
      0 &\text{ otherwise.}
    \end{cases}
  \end{equation*}
  Replacing $\vartheta$ by $\vartheta'$ in \eqref{eq:A4_sum_87_def_sigma} does
  not change $\Sigma$ for any $\ee'' \in \OO_{1*}\times\cdots\times\OO_{6*}$
  and $\kc$ as in \eqref{eq:A4_sum_87_M87}, since the sum over $\rho$ is zero
  whenever $\eI_5+\eI_7 \neq \OO_K$. Indeed, we know from Lemma
  \ref{lem:A4_first_summation} that $\e_6\e_7A$ is invertible modulo $\kc\eI_5$
  whenever $\kc$ is as in \eqref{eq:A4_sum_87_M87} and $\theta_0(\eII') \neq
  0$. This implies that $v_\p(\e_6A\OO_7) = 0$ for any fixed $\ee'', \kc$ as in
  \eqref{eq:A4_sum_87_M87} with $\Sigma \neq 0$ and any $\p \mid
  \kc\eI_5$. Therefore, if $p \mid \eI_5+\eI_7$ then the second and third
  condition under the sum over $\rho$ in \eqref{eq:A4_sum_87_def_sigma}
  contradict each other.

  Since $\vartheta'(\eI_7) = \vartheta(\eI_7)$ whenever $\eI_5+\eI_7 = \OO_K$,
  we have $\mathcal{A}(\vartheta'(\aaa), \aaa, \kc\eI_5) =
  \mathcal{A}(\vartheta(\aaa), \aaa, \kc\eI_5))$. 

  Moreover, we obtain immediately from the definition that $\vartheta' \in
  \Theta(\eI_1\eI_2\eI_3\eI_4,1,1,1)$ (see \cite[Definition
  2.1]{arXiv:1302.6151}). Hence, by \cite[Lemma 2.2]{arXiv:1302.6151}, the
  function $\vartheta'$ satisfies \cite[(6.1)]{arXiv:1302.6151} with $c_\theta
  := 2^{\omega_K(\eI_1\eI_2\eI_3\eI_4)}$, $C:=0$.
  
  By \eqref{eq:A4_height_4}, $g(t)=0$ whenever $t > t_2 :=
  B/(\N\eI_1\N\eI_2^2\N\eI_3^2\N\eI_4^2\N\eI_5\N\eI_6^2)$, and, by \cite[Lemma
  3.5, (2)]{arXiv:1302.6151} applied to \eqref{eq:A4_height_5}, $g(t) \ll B^{1/2}/(\N\eI_1^{1/2}\N\eI_5^{1/2})\cdot
  t^{-1/2}$. With \cite[Proposition 6.1]{arXiv:1302.6151}, we obtain
  \begin{align*}
    \Sigma &= \frac{2 \pi}{\sqrt{|\Delta_K|}}\phi_K^*(\kc\eI_5)\mathcal{A}(\vartheta(\aaa), \aaa, \kc\eI_5)\int_{t \geq\N\eI_6}g(t)\dd t\\
    &+ O\left(\frac{2^{\omega_K(\eI_1\eI_2\eI_3\eI_4)}B^{1/2}}{\N\eI_1^{1/2}\N\eI_5^{1/2}}\left(\sqrt{\N(\kc\eI_5)}\log
        B +
        \frac{\N\kc\eI_5}{\N\eI_6^{1/2}}\log(\N\eI_6+2)\right)\right)\text.
  \end{align*}
  As in Lemma \ref{lem:A4_second_summation_Ma}, the main term in the lemma is
  correct, and for the error term we may sum over the ideals $\kc$ and $\eI_j$ instead of
  the $\e_j$. By \eqref{eq:A4_height_1}, \eqref{eq:A4_height_4}, and our
  condition $\N\eI_7 \geq \N\eI_6$, it suffices to sum over $\kc$ and the
  $(\eI_1, \ldots, \eI_6)$ satisfying \eqref{eq:A4_height_1} and
  \begin{equation}\label{eq:A4_first_fourth_height_cond_I_6}
    \N\eI_1^3\N\eI_2^6\N\eI_3^5\N\eI_4^4\N\eI_5^4\N\eI_6^4 \leq B^2\text.
  \end{equation}
  Thus, the total error is bounded by
  \begin{align*}
    &\sum_{\substack{\eI_1, \ldots,
        \eI_6\\\eqref{eq:A4_first_fourth_height_cond_I_6}}}\left(\frac{2^{\omega_K(\eI_2) + \omega_K(\eI_1\eI_2\eI_3\eI_4)} B^{1/2}\log B}{\N\eI_1^{1/2}} + \frac{2^{\omega_K(\eI_2) + \omega_K(\eI_1\eI_2\eI_3\eI_4)} \N\eI_5^{1/2}B^{1/2}\log B}{\N\eI_1^{1/2}\N\eI_6^{1/2}}\right)\\
    &\ll \sum_{\substack{\eI_1, \ldots, \eI_5\\\N\eI_j \leq
        B}}\frac{2^{\omega_K(\eI_2) +\omega_K(\eI_1\eI_2\eI_3\eI_4)}B\log
      B}{\N\eI_1^{5/4}\N\eI_2^{3/2}\N\eI_3^{5/4}\N\eI_4\N\eI_5} +
    \sum_{\substack{\eI_1, \ldots, \eI_4, \eI_6\\\N\eI_j \leq B}} \frac{2^{\omega_K(\eI_2) + \omega_K(\eI_1\eI_2\eI_3\eI_4)}B\log B}{\N\eI_1^{3/2}\N\eI_2^2\N\eI_3^{3/2}\N\eI_4\N\eI_6}\\
    &\ll B(\log B)^4 + B(\log B)^4 \ll B(\log B)^4\text.\qedhere
  \end{align*}
\end{proof}

\begin{lemma}\label{lem:A4_second_summation_Mb_ideals}
  If $\eII''$ runs over all six-tuples $(\eI_1, \ldots, \eI_6)$ of
  nonzero ideals of $\OO_K$ then we have
  \begin{equation*}
    N_{87}(B) =
    \left(\frac{2}{\sqrt{|\Delta_K|}}\right)^2\sum_{\eII''}\mathcal{A}(\theta_8'(\eII'),
    \eI_7)V_{87}(\N\eI_1,\ldots,\N\eI_6; B) + O(B(\log B)^4)\text.
   \end{equation*}
\end{lemma}

\begin{proof}
  This is analogous to \cite[Lemma 9.4]{arXiv:1302.6151}.
\end{proof}

\subsubsection{The remaining summations for $N_{87}(B)$}
\begin{lemma}\label{lem:A4_completion_Mb}
  We have
  \begin{equation*}
    N_{87}(B) = \pi^6\left(\frac{2}{\sqrt{|\Delta_K|}}\right)^8 \left(\frac{h_K}{\omega_K}\right)^6 \theta_0V_{870}(B) + O(B(\log B)^4\log \log B),
  \end{equation*}
  where $\theta_0$ is given in \eqref{eq:def_theta_0} and
  \begin{equation*}
    V_{870}(B) := \int_{t_1, \dots, t_6 \ge 1}V_{87}(t_1, \dots, t_6; B) \dd t_1 \cdots \dd t_6,
  \end{equation*}
  with real variables $t_1, \ldots, t_6$.
\end{lemma}

\begin{proof}
  By \cite[Lemma 3.5, (6)]{arXiv:1302.6151}, applied to \eqref{eq:A4_height_5},
  we have
  \begin{equation*}
    V_{87}(t_1, \ldots, t_6; B) \ll \frac{1}{t_5}\cdot\frac{B^{3/4}t_5^{3/4}}{t_1^{1/2}t_3^{1/4}t_4^{1/2}t_6^{3/4}}
    =\frac{B}{t_1 \cdots
      t_6}\left(\frac{B}{t_1^2t_2^4t_3^3t_4^2t_5^3t_6}\right)^{-1/4}\text.
  \end{equation*}
  Furthermore, using \eqref{eq:A4_height_3} and \eqref{eq:A4_height_4} to bound
  $\abs{\e_8}$ and $t_7$, respectively, we see that
  \begin{equation*}
    V_{87}(t_1, \ldots, t_6; B) \ll \frac{1}{t_5}\left(\frac{B}{t_1^2t_2^3t_3^2t_4t_5^2}\right)\left(\frac{B}{t_1t_2^2t_3^2t_4^2t_5t_6^2}\right)
    =\frac{B}{t_1 \cdots
      t_6}\cdot\left(\frac{B}{t_1^2t_2^4t_3^3t_4^2t_5^3t_6}\right)\text.
  \end{equation*}
  We apply \cite[Proposition 7.3]{arXiv:1302.6151} with $r=5$.
\end{proof}

\subsubsection{Combining the summations}
\begin{lemma}\label{lem:A4_completion}
  We have
  \begin{equation*}
    N_{U_2,H}(B) = \left(\frac{2}{\sqrt{|\Delta_K|}}\right)^8 \left(\frac{h_K}{\omega_K}\right)^6 \theta_0V_0(B) + O(B(\log B)^4\log \log B),
  \end{equation*}
  where $\theta_0$ is given in \eqref{eq:def_theta_0} and
  \begin{equation*}
    V_0(B) := \int\limits_{\substack{(\e_1, \ldots, \e_8) \in \mathcal{R}(B)\\\abs{\e_1}, \ldots, \abs{\e_7} \ge
        1}}\frac{1}{\abs{\e_5}}\dd \e_1 \cdots \dd \e_8,
  \end{equation*}
 with complex variables $\e_1, \ldots, \e_8$.
\end{lemma}

\begin{proof}
  This follows from \eqref{eq:A4_split_summation}, Lemma
  \ref{lem:A4_completion_Ma}, and Lemma \ref{lem:A4_completion_Mb},
  using polar coordinates, similarly to \cite[Lemma 9.9]{arXiv:1302.6151}.
\end{proof}

\subsection{Proof of Theorem \ref{thm:main} for $S_2$}
We use the conditions
\begin{align}
  &\abs{\e_1^2\e_2^4\e_4^2\e_5^3\e_6}\leq B\text{ and }\label{eq:A4_comparison_1}\\
  &\abs{\e_1^2\e_2^4\e_4^2\e_5^3\e_6}\leq B \text{ and
  }\abs{\e_1^{-1}\e_2^{-2}\e_4^{2}\e_5^{-3}\e_6^{4}}\le
  B.\label{eq:A4_comparison_2}
\end{align}
\begin{lemma}\label{lem:A4_predicted_volume}
  Let $\alpha(\tS_2)$, $\omega_\infty(\tS_2)$ be as in Theorem \ref{thm:main}
  and $\mathcal{R}(B)$ as in \eqref{eq:A4_height_1}--\eqref{eq:A4_height_5}. Define
  \begin{equation*}
    V_0'(B) := \int_{\substack{(\e_1, \ldots, \e_8)\in\mathcal{R}(B)\\\abs{\e_1}, \abs{\e_2}, \abs{\e_4}, \abs{\e_5},
      \abs{\e_6}\geq 1\\\eqref{eq:A4_comparison_2}}}\frac{1}{\abs{\e_5}} \dd
  \e_1 \cdots \dd \e_8,
  \end{equation*}
  where $\e_1$, $\ldots$, $\e_8$ are complex variables. Then
  \begin{equation}\label{eq:A4_predicted_volume}
    \pi^6\alpha(\tS_2) \omega_\infty(\tS_2) B(\log B)^5 = 4 V_0'(B).
  \end{equation}
\end{lemma}

\begin{proof}
  The proof is analogous to the proof of Lemma
  \ref{lem:A3+A1_predicted_volume}. Let $\e_1, \e_2, \e_4, \e_5, \e_6
  \in \CC$, $B > 0$, and let $l :=
  (B\abs{\e_1\e_2^2\e_4\e_5^3\e_6^2})^{1/2}$. Let $\e_3, \e_7, \e_8$ be
  complex variables. Applying the coordinate transformation $z_0 =
  l^{-1/3}\e_1\e_2^2\e_4\e_5^2\e_6\cdot \e_3$, $z_2 = l^{-1/3}\e_1\e_2\e_5
  \cdot \e_8$, $z_3 = l^{-1/3}\e_4\e_6^2 \cdot\e_7$ to $\omega_\infty(\tS_2)$,
  we obtain
  \begin{equation}\label{eq:A4_complex_density_torsor}
    \omega_\infty(\tS_2) = \frac{12}{\pi}\frac{\abs{\e_1\e_2\e_4\e_5\e_6}}{B} \int_{(\e_1, \ldots, \e_8)\in\mathcal{R}(B)}\frac{1}{\abs{\e_5}}\dd \e_3 \dd \e_7 \dd \e_8\text.
  \end{equation}

  The negative curves $[E_1], \dots, [E_7]$ generate the effective cone of
  $\tS_1$. Because of $[-K_{\tS_1}] = [2E_1+4E_2+3E_3+2E_4+3E_5+E_6]$ and $[E_7]
  = [E_1+2E_2+E_3+2E_5-E_6]$, \cite[Lemma 8.1]{arXiv:1302.6151} implies
  \begin{equation}\label{eq:A4_alpha}
      \alpha(\tS_2)(\log B)^5=\frac 1{3\pi^5} \int_{\substack{\abs{\e_1}, \abs{\e_2},
          \abs{\e_4}, \abs{\e_5}, \abs{\e_6}\geq 1\\\eqref{eq:A4_comparison_2}}}
      \frac{\dd \e_1\dd \e_2\dd \e_4\dd \e_5\dd
        \e_6}{\abs{\e_1\e_2\e_4\e_5\e_6}}\text.
  \end{equation}
  The lemma follows by substituting
  \eqref{eq:A4_complex_density_torsor} and \eqref{eq:A4_alpha} in
  \eqref{eq:A4_predicted_volume}.
\end{proof}

To finish our proof, we compare $V_0(B)$ from Lemma
\ref{lem:A4_completion} with $V_0'(B)$ defined in Lemma
\ref{lem:A4_predicted_volume}. Let 
\begin{align*}
  \mathcal{D}_0(B) &:= \{(\e_1, \ldots, \e_8) \in \mathcal{R}(B) \mid \abs{\e_1},\ldots,\abs{\e_7} \geq 1\},\\
  \mathcal{D}_1(B) &:= \{(\e_1, \ldots, \e_8) \in \mathcal{R}(B) \mid \abs{\e_1},\ldots,\abs{\e_7} \geq 1\text{, }\eqref{eq:A4_comparison_1}\},\\
  \mathcal{D}_2(B) &:= \{(\e_1, \ldots, \e_8) \in \mathcal{R}(B) \mid \abs{\e_1},\ldots,\abs{\e_7} \geq 1\text{, }\eqref{eq:A4_comparison_2}\},\\
  \mathcal{D}_3(B) &:= \{(\e_1, \ldots, \e_8) \in \mathcal{R}(B) \mid \abs{\e_1},\ldots,\abs{\e_6} \geq 1\text{, }\eqref{eq:A4_comparison_2}\},\\
  \mathcal{D}_4(B) &:= \{(\e_1, \ldots, \e_8) \in \mathcal{R}(B) \mid \abs{\e_1},
  \!\abs{\e_2},\!\abs{\e_4},\!\abs{\e_5},\!\abs{\e_6} \geq 1\text{, }\eqref{eq:A4_comparison_2}\}.
\end{align*}
Moreover, let
\begin{equation*}
  V_i(B) := \int_{\mathcal{D}_i(B)}\frac{\dd \e_1 \cdots \dd \e_8}{\abs{\e_5}}\text.
\end{equation*}
Then clearly $V_0(B)$ is as in Lemma \ref{lem:A4_completion} and
$V_4(B) = V_0'(B)$. We show that, for $1 \leq i \leq 4$, $V_i(B) -
V_{i-1}(B) = O(B(\log B)^4)$. This holds for $i = 1$, since $R_1 =
R_0$. Moreover, using \cite[Lemma 3.5, (4)]{arXiv:1302.6151} and \eqref{eq:A4_height_5} to
bound the integral over $\e_7$ and $\e_8$, we have
\begin{equation*}
  V_2(B) - V_1(B) \ll \int_{\substack{\abs{\e_1},\ldots,\abs{\e_6}\geq 1\\\abs{\e_1\e_2^2\e_3^2\e_4^2\e_5\e_6^2}\leq B\\\abs{\e_1^{-1}\e_2^{-2}\e_4^2\e_5^{-3}\e_6^4} > B}}\frac{B^{3/4}}{\abs{\e_1^2\e_3\e_4^2\e_5\e_6^{3}}^{1/4}}\dd \e_1 \cdots \dd \e_6 \ll B(\log B)^4\text.
\end{equation*}
Using \cite[Lemma 3.5, (2)]{arXiv:1302.6151} and the \eqref{eq:A4_height_5} to
bound the integral over $\e_8$, we obtain
\begin{equation*}
  V_3(B) - V_2(B) \ll \int\limits_{\substack{\abs{\e_1},\ldots,\abs{\e_6}\geq 1\\\abs{\e_7} < 1\text{,
      }\eqref{eq:A4_height_1}\text{, }\eqref{eq:A4_comparison_2}}}\frac{B^{1/2}}{\abs{\e_1\e_5\e_7}^{1/2}}\dd \e_1 \cdots \dd \e_7 \ll B(\log B)^4\text.  
\end{equation*}
Finally, using using \cite[Lemma 3.5, (4)]{arXiv:1302.6151} and
\eqref{eq:A4_height_5} to bound the integral over $\e_7$ and $\e_8$, we have
\begin{equation*}
  V_4(B) - V_3(B) \ll\hspace{-1.3cm} \int\limits_{\substack{\abs{\e_1},\abs{\e_2},\abs{\e_4},\abs{\e_5},\abs{\e_6}\geq 1\\\abs{\e_3} < 1\text{, }\eqref{eq:A4_comparison_1}}}\frac{B^{3/4}}{\abs{\e_1^2\e_3\e_4^2\e_5\e_6^3}^{1/4}}\dd \e_1 \cdots \dd \e_6 \ll B(\log B)^4\text.
\end{equation*}
Using Lemma~\ref{lem:A4_completion} and Lemma~\ref{lem:A4_predicted_volume},
this implies Theorem \ref{thm:main} for $S_2$.

\section{The quartic del Pezzo surface of type $\Dfour$}

\subsection{Passage to a universal torsor}
We use the notation from \cite{math.AG/0604194}.

\begin{figure}[ht]
  \centering
  \[\xymatrix@R=0.05in @C=0.05in{E_8  \ar@{-}[r]\ar@{-}[dr] \ar@{=}[dd]& \li{5} \ar@{-}[r] & \ex{3} \ar@{-}[dr]\\
  & E_7 \ar@{-}[r] & \ex{2} \ar@{-}[r] & \ex{1} \\
  E_9 \ar@{-}[r] \ar@{-}[ur] & \li{6} \ar@{-}[r] & \ex{4} \ar@{-}[ur]}\]
  \caption{Configuration of curves on $\tS_3$}
  \label{fig:D4_dynkin}
\end{figure}
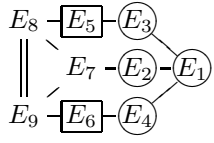

For any given $\classtuple = (C_0, \dots, C_5) \in \classrep^6$, we
define $u_\classtuple := \N(C_0^3 C_1^{-1}\cdots C_5^{-1})$ and
\begin{align*}
    \OO_1 &:= C_2C_3^{-1} & \OO_2 &:= C_1C_2^{-1} & \OO_3 &:= C_0C_1^{-1}C_2^{-1}C_5^{-1}\\
    \OO_4 &:= C_3C_4^{-1} & \OO_5 &:= C_5 & \OO_6 &:= C_4 \\
    \OO_7 &:= C_0 C_1^{-1} & \OO_8 &:= C_0 C_5^{-1} & \OO_9 &:= C_0^2
    C_1^{-1} C_2^{-1} C_3^{-1} C_4^{-1}
\end{align*}
Let
\begin{equation*}
  \OO_{j*} :=
  \begin{cases}
    \OO_j^{\neq 0}, & j \in \{1,\ldots, 6\},\\
    \OO_j, & j \in \{7,8,9\}.
  \end{cases}
\end{equation*}
For $\eta_j \in \OO_j$, let
\begin{equation*}
  \eI_j := \e_j \OO_j^{-1}\text.
\end{equation*}
For $B \geq 0$, let $\mathcal{R}(B)$ be the set of all $(\e_1, \ldots, \e_8)
\in \CC^8$ with $\e_4\e_6 \neq 0$ and
\begin{align}
  \abs{\e_1^2\e_2\e_3^2\e_4\e_5^2\e_8}&\leq B,\label{eq:D4_height_1}\\
  \abs{\e_1^4\e_2^2\e_3^3\e_4^3\e_5^2\e_6^2}&\leq B,\label{eq:D4_height_2}\\
  \abs{\e_1^3\e_2^2\e_3^2\e_4^2\e_5\e_6\e_7}&\leq B,\label{eq:D4_height_3}\\
  \abs{\e_1^2\e_2\e_3^2\e_4\e_5^2\e_8 + \e_1^2\e_2^2\e_3\e_4\e_7^2}&\leq B,\label{eq:D4_height_4}\\
  \abs{\frac{\e_3\e_5^2\e_8^2+\e_2\e_7^2\e_8}{\e_4\e_6^2}}&\leq B,\label{eq:D4_height_5}
\end{align}
and let $M_\classtuple(B)$ be the set of all
\begin{equation*}
  (\e_1, \ldots, \e_9) \in \OO_{1*} \times \cdots \times \OO_{9*} 
\end{equation*}
that satisfy the \emph{height conditions}
\begin{equation*}\label{eq:D4_height}
  (\e_1, \ldots, \e_8) \in \mathcal{R}(u_\classtuple B)\text,
\end{equation*}
the \emph{torsor equation}
\begin{equation}\label{eq:D4_torsor}
  \e_2\e_7^2+\e_3\e_5^2\e_8 + \e_4\e_6^2\e_9 = 0,
\end{equation}
and the \emph{coprimality conditions}
\begin{equation}\label{eq:D4_coprimality}
  \eI_j + \eI_k = \OO_K \text{ for all distinct nonadjacent vertices $E_j$, $E_k$ in Figure~\ref{fig:D4_dynkin}.}
\end{equation}

\begin{lemma}\label{lem:D4_passage_to_torsor}
  We have
  \begin{equation*}
    N_{U_3,H}(B) = \frac{1}{\omega_K^6}\sum_{\classtuple \in \classrep^6}|M_\classtuple(B)|.
  \end{equation*}
\end{lemma}

\begin{proof}
  Again, the lemma is a specialization of \cite[Claim 6.1]{arXiv:1302.6151},
  and we prove it in an analogous way as Lemma
  \ref{lem:A4_passage_to_torsor}. Starting with the curves
  Let $E_3^{(0)} := \{y_1=0\}$, $E_7^{(0)} := \{y_2=0\}$, $E_8^{(0)} := \{y_0 = 0\}$,
  $E_9^{(0)} := \{-y_0y_1-y_2^2 = 0\}$ in $\PP_K^2$, we prove \cite[Claim
  6.2]{arXiv:1302.6151} for the following sequence of blow-ups:
  \begin{enumerate}
  \item blow up $E_3^{(0)}\cap E_7^{(0)}\cap E_9^{(0)}$, giving $E_2^{(1)}$,
  \item blow up $E_2^{(1)}\cap E_3^{(1)}\cap E_9^{(1)}$, giving $E_1^{(2)}$,
  \item blow up $E_1^{(2)}\cap E_9^{(2)}$, giving $E_4^{(3)}$,
  \item blow up $E_4^{(3)} \cap E_9^{(3)}$, giving $E_6^{(4)}$,
  \item blow up $E_3^{(4)} \cap E_8^{(4)}$, giving $E_5^{(5)}$.
  \end{enumerate}
  
  The inverse $\pi \circ \rho^{-1} : \PP^2_K \rto S_3$ of the projection $\rho \circ \pi^{-1} :
  S_3 \rto \PP^2_K$, $(x_0 : \cdots : x_4) \mapsto (x_0 : x_1
  : x_2)$ is given by
  \begin{equation*}
    (y_0 : y_1 : y_2) \mapsto (y_0 y_1^2 : y_1^3 : y_1^2y_2 : -y_1(y_0 y_1 + y_2^2) : -y_0(y_0 y_1 + y_2^2))\text.
  \end{equation*}
  With the map $\Psi$ from \cite[Claim 4.2]{arXiv:1302.6151} sending $(\e_1,
  \ldots, \e_9)$ to 
  \begin{equation*}
    (\e_1^2\e_2\e_3^2\e_4\e_5^2\e_8, \e_1^4\e_2^2\e_3^3\e_4^3\e_5^2\e_6^2, \e_1^3\e_2^2\e_3^2\e_4^2\e_5\e_6\e_7, \e_1^2\e_2\e_3\e_4^2\e_6^2\e_9, \e_8\e_9),
  \end{equation*}
  we see that the assumptions of \cite[Lemma 4.3]{arXiv:1302.6151} are
  satisfied, so \cite[Claim 4.2]{arXiv:1302.6151} holds for $i=0$.
  
  In the first two steps of the above chain of blow-ups, we are in the
  situation of \cite[Remark 4.5]{arXiv:1302.6151}, so certain coprimality
  conditions need to be checked by hand. However, up to changing some indices,
  our situation in steps (1) and (2) is exactly the same as in Lemma
  \ref{lem:A4_passage_to_torsor}, so the arguments given there apply to our
  lemma as well. Steps (3), (4), (5) are again covered by
  \cite[Lemma~4.4]{arXiv:1302.6151}, which proves \cite[Claim
  6.2]{arXiv:1302.6151}. From this, we deduce \cite[Claim 6.1]{arXiv:1302.6151}
  as in \cite[Lemma 9.1]{arXiv:1302.6151}.
\end{proof}

\subsection{Summations}
\subsubsection{The first summation over $\e_8$ with dependent $\e_9$}
\begin{lemma}\label{lem:D4_first_summation}
  Let $\ee' := (\e_1, \ldots, \e_7)$ and $\eII' := (\eI_1, \ldots, \eI_7)$. Then
  \begin{equation*}
    |M_\classtuple(B)| = \frac{2}{\sqrt{|\Delta_K|}} \sum_{\ee' \in \OO_{1*} \times \dots \times \OO_{7*}} \theta_8(\eII')V_8(\N\eI_1. \ldots, \N\eI_7; B) + O_\classtuple(B(\log B)^2),
  \end{equation*}
  where
  \begin{equation*}
    V_8(t_1, \ldots, t_7; B) :=  \frac{1}{t_4 t_6^2}\int_{(\sqrt{t_1}, \ldots,
      \sqrt{t_7},\e_8) \in \mathcal{R}(B)} \dd \e_8
  \end{equation*}
  and
  \begin{equation*}
    \theta_8(\eII') := \prod_{\id p} \theta_{1,\id p}(J_{\id p}(\eII')). 
  \end{equation*}
  Here, $J_{\id p}(\eII') := \{j \in \{1, \dots, 7\}\ :\ \id p \mid
  \eI_j\}$ and
  \begin{equation*}
    \theta_{1,\id p}(J) :=
    \begin{cases}
      1 &\text{ if }J = \emptyset,\{5\},\{6\},\{7\},\\
      1-\frac{1}{\N\id p} &\text{ if }J = \{2\},\{3\},\{4\},\{1,2\},\{1,3\},\{1,4\},\{2,7\},\{3,5\},\{4,6\},\\
      1-\frac{2}{\N\id p} &\text{ if }J = \{1\},\\
      0 &\text{otherwise.}
    \end{cases}
  \end{equation*}
\end{lemma}

\begin{proof}
  We apply \cite[Proposition 5.3]{arXiv:1302.6151} with $(A_1, A_0) := (2,7)$,
  $(B_1, B_2, B_0) := (3,5,8)$, $(C_1, C_2, C_0):=(4,6,9)$, $D:= 1$,
  $u_\classtuple B$ instead of $B$, and $\Ao$,
  $\At$ as suggested in \cite[Remark 5.2]{arXiv:1302.6151}.

  As in Lemma \ref{lem:A3+A1_first_summation}, we see that the main arising
  from \cite[Proposition 5.3]{arXiv:1302.6151} is the main term in the lemma,
  so it remains to deal with the error term.

  For given $\ee'$ and $B$, the set of all $\e_8 \in
  \CC$ with $(\e_1, \ldots, \e_8)\in\mathcal{R}(u_\classtuple B)$ is contained
  in the union of two balls of radius 
  \begin{equation*}
    R(\ee'; u_\classtuple B) \ll_\classtuple
    \begin{cases}
      (B\N(\eI_4 \eI_6^2 \eI_2^{-1} \eI_7^{-2}))^{1/2} &\text{ if }\e_7 \neq 0\text,\\
      (B/\N(\eI_1^{2}\eI_2\eI_3^{2}\eI_4\eI_5^{2}))^{1/2} &\text{ if
      }\e_7 = 0\text.
    \end{cases}
  \end{equation*}
  Indeed, this follows from \cite[Lemma 3.4,
  (1)]{arXiv:1302.6151}, applied to \eqref{eq:D4_height_5}, if $\e_7\neq 0$ and
  from \eqref{eq:D4_height_1} if $\eta_7 = 0$.

  Hence, the error term is
  \begin{equation*}
    \ll \sum_{\ee'\text{, }\eqref{eq:second_height_cond_ideals}\text{, }\eqref{eq:third_height_cond_ideals}}2^{\omega_K(\eI_1\eI_4)+\omega_K(\eI_1\eI_2\eI_3)}\left(\frac{R(\ee'; u_\classtuple B)}{\N(\eI_4 \eI_6^2)^{1/2}}+1\right)\text,
  \end{equation*}
  where, using \eqref{eq:D4_height_2}, \eqref{eq:D4_height_3}, the sum runs over all $\ee'$ with
  \begin{align}
    \N(\eI_1^4 \eI_2^2\eI_3^3\eI_4^3\eI_5^2 \eI_6^2) &\leq B \label{eq:second_height_cond_ideals}\text{, and}\\
    \N(\eI_1^3\eI_2^2\eI_3^2\eI_4^2\eI_5\eI_6\eI_7) &\leq
    B \label{eq:third_height_cond_ideals}\text.
  \end{align}
  Let us first estimate the sum over all $\ee'$ with $\e_7 \neq 0$. We may sum over the
  $I_j$ instead of the $\e_j$ and obtain
    \begin{align*}
      &\ll_\classtuple \sum_{\eII'\text{, }\eqref{eq:second_height_cond_ideals}\text{, }\eqref{eq:third_height_cond_ideals}}2^{\omega_K(\eI_1\eI_4)+\omega_K(\eI_1\eI_2\eI_3)}\left(\frac{B^{1/2}}{\N(\eI_2\eI_7^2)^{1/2}}+1\right)\\
      &\ll\sum_{\substack{\eI_1, \dots, \eI_5, \eI_7\\\N\eI_j \leq B}}\left(\frac{2^{\omega_K(\eI_1\eI_4)+\omega_K(\eI_1\eI_2\eI_3)}B}{\N\eI_1^2\N\eI_2^{3/2}\N\eI_3^{3/2}\N\eI_4^{3/2}\N\eI_5\N\eI_7}+\frac{2^{\omega_K(\eI_1\eI_4)+\omega_K(\eI_1\eI_2\eI_3)}B}{\N\eI_1^3\N\eI_2^2\N\eI_3^2\N\eI_4^2\N\eI_5\N\eI_7}\right)\\
      &\ll B(\log B)^2.
    \end{align*}
  
  Now we assume that $\e_7 = 0$ and sum over the remaining variables. We obtain
  the upper bound
  \begin{align*}
      &\ll_\classtuple\sum_{\eI_1, \ldots, \eI_6\text{, }\eqref{eq:second_height_cond_ideals}}
      2^{\omega_K(\eI_1\eI_4)+\omega_K(\eI_1\eI_2\eI_3)}\left(\frac{B^{1/2}}{\N\eI_1\N\eI_2^{1/2}\N\eI_3\N\eI_4\N\eI_5\N\eI_6}+1\right)\\
      &\ll\sum_{\eI_1, \eI_3, \ldots, \eI_6}\left(\frac{2^{\omega_K(\eI_1\eI_4)+\omega_K(\eI_1\eI_3)}B^{3/4}\log B}{\N\eI_1^{2}\N\eI_3^{7/4}\N\eI_4^{7/4}\N\eI_5^{3/2}\N\eI_6^{3/2}} + \frac{2^{\omega_K(\eI_1\eI_4)+\omega_K(\eI_1\eI_3)}B^{1/2}\log B}{\N\eI_1^{2}\N\eI_3^{3/2}\N\eI_4^{3/2}\N\eI_5\N\eI_6} \right)\\
      &\ll B^{3/4}\log B + B^{1/2}(\log B)^3\ll B^{3/4}\log B\text.\qedhere
  \end{align*}
\end{proof}
  
\subsubsection{The second summation over $\e_7$}
\begin{lemma}\label{lem:D4_second_summation}
  Write $\ee'' := (\e_1, \dots, \e_6)$. We have
  \begin{align*}
    |M_\classtuple(B)| &= \left(\frac{2}{\sqrt{|\Delta_K|}}\right)^2\sum_{\ee''\in \OO_{1*} \times \dots \times \OO_{6*}} \mathcal{A}(\theta_8(\eII'), \eI_7)V_7(\N\eI_1, \ldots, \N\eI_6; B)\\ &+ O_\classtuple(B(\log B)^2),
  \end{align*}
  where, for $t_1, \ldots, t_6 \geq 1$,
  \begin{equation*}
    V_7(t_1, \ldots, t_6; B) := \frac{\pi}{t_4 t_6^2}\int_{(\sqrt{t_1}, \ldots,
      \sqrt{t_7}, \e_8) \in \mathcal{R}(B)} \dd t_7 \dd \e_8,
  \end{equation*}
 with a positive variable $t_7$ and a complex variable $\e_8$.
\end{lemma}

\begin{proof}
  We apply \cite[Proposition 6.1]{arXiv:1302.6151} as suggested in
  \cite[Section 6]{arXiv:1302.6151} in the case $b_0=1$. We have
  \begin{equation}\label{eq:D4_sec_sum_start}
    |M_\classtuple(B)| = \frac{2}{\sqrt{|\Delta_K|}}\sum_{\ee'' \in \OO_{1*} \times \dots \times \OO_{6*}}\sum_{\e_7 \in \OO_7}\vartheta(\eI_7)g(\N\eI_7) + O_\classtuple(B(\log B)^2)\text,
  \end{equation}
  where $\vartheta(\eI_7) := \theta_8(\eII')$ and
  $g(t) := V_8(\N\eI_1, \ldots, \N\eI_6, t; B)$.

  By \cite[Lemma 5.4, Lemma 2.2]{arXiv:1302.6151},  the function $\vartheta$
  satisfies \cite[(6.1)]{arXiv:1302.6151} with $C := 0$ and $c_\vartheta :=
  2^{\omega_K(\eI_1\eI_3 \cdots \eI_6)}$.
 
  By \eqref{eq:D4_height_3}, we have $g(t) = 0$ whenever $t > t_2 := B/(\N\eI_1^3\N\eI_2^2\N\eI_3^2\N\eI_4^2
  \N\eI_5\N\eI_6)$, and by \cite[Lemma 3.4, (2)]{arXiv:1302.6151} applied to
  \eqref{eq:D4_height_5}, we obtain
  \begin{equation*}
    g(t) \ll \frac{1}{\N\eI_4\N\eI_6^2}\cdot \frac{(\N\eI_4\N\eI_6^2 B)^{1/2}}{(\N\eI_3\N\eI_5^2)^{1/2}} = \frac{B^{1/2}}{\N\eI_3^{1/2}\N\eI_4^{1/2}\N\eI_5\N\eI_6} =: c_g\text.
  \end{equation*}
  By \cite[Proposition 6.1]{arXiv:1302.6151}, the sum over $\e_7$ in \eqref{eq:D4_sec_sum_start} is just 
  \begin{align*}
    &\vartheta((0))g(0) + \frac{2\pi}{\sqrt{|\Delta_K|}}
    \mathcal{A}(\vartheta(\aaa), \aaa, \OO_K) \int_{t \geq 1}g(t)\dd t\\ &+
    O\left(\frac{2^{\omega_K(\eI_1\eI_3\cdots\eI_6)}B^{1/2}}{\N\eI_3^{1/2}\N\eI_4^{1/2}\N\eI_5\N\eI_6}\cdot\frac{B^{1/2}}{\N\eI_1^{3/2}\N\eI_2\N\eI_3\N\eI_4\N\eI_5^{1/2}\N\eI_6^{1/2}}\right).
  \end{align*}

  Due to (\ref{eq:D4_height_2}), $\vartheta((0))g(0)$ and $\int_0^1g(t)\dd t$ are dominated
  by the error term, so the main term in the lemma is correct.

  Let us consider the error term. Both the sum and the integral are
  zero whenever $\ee''$ violates \eqref{eq:D4_height_2}. We may sum over the
  $(\eI_1, \ldots, \eI_6)$ satisfying \eqref{eq:second_height_cond_ideals} instead of
  the $\ee''$, so the error term is
  \begin{align*}
      &\ll \sum_{\eII'',\ \N\eI_j\leq B}2^{\omega_K(\eI_1\eI_3\cdots\eI_6)}\left(\frac{B}{\N\eI_1^{3/2}\N\eI_2\N\eI_3^{3/2}\N\eI_4^{3/2}\N\eI_5^{3/2}\N\eI_6^{3/2}}\right)\\
      &\ll B\log B.\qedhere
\end{align*}
\end{proof}

\begin{lemma}\label{lem:D4_second_summation_ideals}
  If $\eII''$ runs over all six-tuples $(\eI_1, \ldots, \eI_6)$ of
  nonzero ideals of $\OO_K$ then we have
  \begin{equation*}
    N_{U_3,H}(B) = \left(\frac{2}{\sqrt{|\Delta_K|}}\right)^2\sum_{\eII''}\mathcal{A}(\theta_8(\eII'), \eI_7) V_7(\N\eI_1, \ldots, \N\eI_6; B) + O(B(\log B)^2)\text. 
  \end{equation*}
\end{lemma}

\begin{proof}
 This is analogous to \cite[Lemma 9.4]{arXiv:1302.6151}.
\end{proof}

\subsubsection{The remaining summations}
\begin{lemma}\label{lem:D4_completion}
  We have
  \begin{equation*}
   N_{U_3,H}(B) = \left(\frac{2}{\sqrt{|\Delta_K|}}\right)^8 \left(\frac{h_K}{\omega_K}\right)^6 \theta_0V_0(B) + O(B(\log B)^4\log \log B),
  \end{equation*}
  where $\theta_0$ is as in \eqref{eq:def_theta_0} and
  \begin{equation*}
    V_0(B) := \int\limits_{\substack{(\e_1, \dots,  \e_8) \in \mathcal{R}(B)\\\abs{\e_1}, \ldots, \abs{\e_6}\ge 1}}\frac{1}{\abs{\e_4 \e_6^2}}\dd \e_1 \cdots \dd \e_8,
  \end{equation*}
  with complex variables $\e_1, \ldots, \e_8$.
\end{lemma}

\begin{proof}
  By \cite[Lemma 3.4, (5)]{arXiv:1302.6151} applied to \eqref{eq:D4_height_5}, we have
  \begin{equation*}
    V_7(t_1, \ldots, t_6; B) \ll \frac{B^{3/4}}{t_2^{1/2}t_3^{1/4}t_4^{1/4}t_5^{1/2}t_6^{1/2}}
    =\frac{B}{t_1 \cdots
      t_6}\left(\frac{B}{t_1^4t_2^2t_3^3t_4^3t_5^2t_6^2}\right)^{-1/4}\text.
  \end{equation*}
  We apply \cite[Proposition 7.3]{arXiv:1302.6151} with $r=5$ and use polar coordinates.
\end{proof}

\subsection{Proof of Theorem \ref{thm:main} for $S_3$}\label{sec:D4_completion}

\begin{lemma}\label{lem:D4_predicted_volume}
  Let $\alpha(\tS_3)$, $\omega_\infty(\tS_3)$ be as in Theorem \ref{thm:main}
  and $\mathcal{R}(B)$ as in \eqref{eq:D4_height_1}--\eqref{eq:D4_height_5}. Define
  \begin{equation*}
    V_0'(B) := \int_{\substack{(\e_1, \ldots, \e_8) \in
        \mathcal{R}(B)\\\abs{\e_1},\abs{\e_2},\abs{\e_4},\abs{\e_5},\abs{\e_6}\ge
      1\\\abs{\e_1^4\e_2^2\e_4^3\e_5^2\e_6^2}\le
      B}}\frac{1}{\abs{\e_4\e_6^2}}\dd \e_1 \cdots \dd \e_8,
  \end{equation*}
  where $\e_1, \ldots, \e_8$ are complex variables. Then
  \begin{equation}\label{eq:D4_predicted_volume}
   \pi^6\alpha(\tS_3) \omega_\infty(\tS_3) B(\log B)^5 = 4 V_0'(B).
  \end{equation}
\end{lemma}

\begin{proof}
  Let $\e_1, \e_2, \e_4, \e_5, \e_6 \in \CC$, $B>0$, and define $l :=
  (B \abs{\e_1^2\e_2\e_4^3\e_5\e_6^4})^{1/2}$. Let $\e_3, \e_7, \e_8$ be
  complex variables. After the coordinate transformation $z_0 = l^{-1/3}
    \e_5\cdot\e_8$, $z_1 = l^{-1/3} \e_1^2\e_2\e_4^2\e_5\e_6^2 \cdot \e_3$,
  $z_2 = l^{-1/3}\e_1\e_2\e_4\e_6\cdot\e_7$, we have
  \begin{equation}\label{eq:D4_complex_density_torsor}
    \omega_\infty(\tS_3) = \frac{12}{\pi}\frac{\abs{\e_1\e_2\e_4\e_5\e_6}}{B} \int_{(\e_1, \ldots, \e_8)\in\mathcal{R}(B)}\frac{1}{\abs{\e_4 \e_6^2}}\dd \e_3 \dd \e_7 \dd \e_8\text.
  \end{equation}

  Since the negative curves $[E_1], \dots, [E_6]$ generate the effective
  cone of $\tS_3$, and $[-K_{\tS_3}] = [4E_1+2E_2+3E_3+3E_4+2E_5+2E_6]$,
  \cite[Lemma 8.1]{arXiv:1302.6151} gives
  \begin{equation}\label{eq:D4_alpha}
    \alpha(\tS_3)(\log B)^5=\frac 1{3\pi^5}
    \int_{\substack{\abs{\e_1},\abs{\e_2},\abs{\e_4},\abs{\e_5},\abs{\e_6}\geq
        1\\\abs{\e_1^4\e_2^2\e_4^3\e_5^2\e_6^2}\leq B}}
    \frac{\dd \e_1\dd \e_2\dd \e_4\dd \e_5\dd
      \e_6}{\abs{\e_1\e_2\e_4\e_5\e_6}}\text.
  \end{equation}
  The lemma follows by substituting
  \eqref{eq:D4_complex_density_torsor} and \eqref{eq:D4_alpha} in
  \eqref{eq:D4_predicted_volume}.
\end{proof}

To finish our proof, we compare $V_0(B)$ from Lemma \ref{lem:D4_completion}
with $V_0'(B)$ defined in Lemma \ref{lem:D4_predicted_volume}. We show that,
starting from $V_0(B)$, we can add the condition
$\abs{\e_1^4\e_2^2\e_4^3\e_5^2\e_6^2} \le B$ and remove $\abs{\e_3} \ge 1$ with
negligible error. First, we note that \eqref{eq:D4_height_2}, together with
$\abs{\e_3} \ge 1$ implies the condition $\abs{\e_1^4\e_2^2\e_4^3\e_5^2\e_6^2}
\le B$, so we can add it to the domain of integration for $V_0(B)$ without
changing the result.

Using \cite[Lemma 3.4, (3)]{arXiv:1302.6151} applied to \eqref{eq:D4_height_5}
to bound the integral over $\e_7$, $\e_8$, we see that an upper bound for
$V_0'(B)-V_0(B)$ is given by
\begin{equation*}
  \ll
  \int_{\substack{\abs{\e_1},\abs{\e_2},\abs{\e_4},\abs{\e_5},\abs{\e_6}\geq 1\\\abs{\e_3} <
      1\text{, }\abs{\e_1^4\e_2^2\e_4^3\e_5^2\e_6^2} \le B}}\frac{B^{3/4}}{\abs{\e_2^2\e_3\e_4\e_5^2\e_6^2}^{1/4}}\dd
  \e_1\cdots\dd\e_6 \ll B(\log B)^4.
\end{equation*}

Using Lemma~\ref{lem:D4_completion} and
Lemma~\ref{lem:D4_predicted_volume}, this implies Theorem \ref{thm:main} for $S_3$.

\section{The quartic del Pezzo surface of type $\Dfive$}

\subsection{Passage to a universal torsor}
We use the notation of \cite{math.AG/0604194}, except that we switch
$\e_7$ with $\e_8$.
\begin{figure}[ht]
  \centering
  \[\xymatrix@R=0.05in @C=0.05in{E_7 \ar@{-}[r]\ar@{=}[dd] \ar@{-}[dr]  & \ex{5} \ar@{-}[r] & \ex{4}\ar@{-}[dr] \\
    & E_8\ar@{-}[r]  & \ex{3} \ar@{-}[r] & \ex{1}\\
    E_9 \ar@{-}[r]\ar@3{-}[ur] & \li{6} \ar@{-}[r] &
    \ex{2}\ar@{-}[ur]}\]
  \caption{Configuration of curves on $\tS_4$}
  \label{fig:D5_dynkin}
\end{figure}
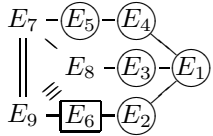

For any given $\classtuple = (C_0, \dots, C_5) \in \classrep^6$, we
define $u_\classtuple := \N(C_0^3 C_1^{-1}\cdots C_5^{-1})$ and
\begin{align*}
    \OO_1 &:= C_3C_4^{-1} & \OO_2 &:= C_4C_5^{-1} & \OO_3 &:= C_0C_1^{-1}C_2^{-1}C_3^{-1}\\
    \OO_4 &:= C_2C_3^{-1} & \OO_5 &:= C_1C_2^{-1} & \OO_6 &:= C_5 \\
    \OO_7 &:= C_0 C_1^{-1} & \OO_8 &:= C_0 & \OO_9 &:= C_0^3 C_1^{-1}
    C_2^{-1} C_3^{-1} C_4^{-1} C_5^{-1}
\end{align*}
Let
\begin{equation*}
  \OO_{j*} :=
  \begin{cases}
    \OO_j^{\neq 0}, & i \in \{1,\ldots, 6\},\\
    \OO_j, & i \in \{7,8,9\}.
  \end{cases}
\end{equation*}
For $\eta_j \in \OO_j$, let
\begin{equation*}
  \eI_j := \e_j \OO_j^{-1}\text.
\end{equation*}
For $B \geq 0$, let $\mathcal{R}(B)$ be the set of all $(\e_1, \ldots, \e_8)
\in \CC^8$ with $\e_2\e_6 \neq 0$ and
\begin{align}
  \abs{\e_1^6\e_2^5\e_3^3\e_4^4\e_5^2\e_6^4}&\leq B,\label{eq:D5_height_1}\\
  \abs{\e_1^2\e_2\e_3\e_4^2\e_5^2\e_7^2}&\leq B,\label{eq:D5_height_2}\\
  \abs{\e_1^4\e_2^3\e_3^2\e_4^3\e_5^2\e_6^2\e_7}&\leq B,\label{eq:D5_height_3}\\
  \abs{\e_1^3\e_2^2\e_3^2\e_4^2\e_5\e_6\e_8}&\leq B,\label{eq:D5_height_4}\\
  \abs{\frac{\e_3\e_8^2 + \e_4\e_5^2\e_7^3}{\e_2\e_6^2}}&\leq B,\label{eq:D5_height_5}
\end{align}
and let $M_\classtuple(B)$ be the set of all
\begin{equation*}
  (\e_1, \ldots, \e_9) \in \OO_{1*} \times \cdots \times \OO_{9*} 
\end{equation*}
that satisfy the \emph{height conditions}
\begin{equation*}\label{eq:D5_height}
  (\e_1, \ldots, \e_8) \in \mathcal{R}(u_\classtuple B)\text,
\end{equation*}
the \emph{torsor equation}
\begin{equation}\label{eq:D5_torsor}
  \e_3\e_8^2 + \e_2\e_6^2\e_9 + \e_4\e_5^2\e_7^3 = 0,
\end{equation}
and the \emph{coprimality conditions}
\begin{equation}\label{eq:D5_coprimality}
  \eI_j + \eI_k = \OO_K \text{ for all distinct nonadjacent vertices $E_j$, $E_k$ in Figure~\ref{fig:D5_dynkin}.}
\end{equation}

\begin{lemma}\label{lem:D5_passage_to_torsor}
  We have
  \begin{equation*}
    N_{U_4,H}(B) = \frac{1}{\omega_K^6}\sum_{\classtuple \in \classrep^6}|M_\classtuple(B)|.
  \end{equation*}
\end{lemma}

\begin{proof}
  The lemma is a special case of \cite[Claim 4.1]{arXiv:1302.6151}, which, as
  before, we prove by proving first \cite[Claim 4.2]{arXiv:1302.6151}, starting
  from the curves $E_3^{(0)}= \{y_0=0\}$, $E_7^{(0)} := \{y_1=0\}$, $E_8^{(0)} := \{y_2 = 0\}$, $E_9^{(0)} :=
  \{-y_0y_2^2 - y_1^3 = 0\}$ in $\PP_K^2$, for the sequence of blow-ups
  \begin{enumerate}
  \item blow up $E_3^{(0)}\cap E_7^{(0)} \cap E_9^{(0)}$, giving $E_5^{(1)}$,
  \item blow up $E_3^{(1)}\cap E_5^{(1)} \cap E_9^{(1)}$, giving $E_4^{(2)}$,
  \item blow up $E_3^{(2)}\cap E_4^{(2)} \cap E_9^{(2)}$, giving $E_1^{(3)}$,
  \item blow up $E_1^{(3)} \cap E_9^{(3)}$, giving $E_2^{(4)}$,
  \item blow up $E_2^{(4)} \cap E_9^{(4)}$, giving $E_6^{(5)}$.
  \end{enumerate}

  With the inverse $\pi\circ\rho^{-1} : \PP^2_K \dasharrow S_4$ of the
  projection $\rho\circ\pi^{-1} : S_4 \rto \PP^2_K$, $(x_0 : \cdots : x_4)
  \mapsto (x_0 : x_2 : x_3)$ given by
  \begin{equation*}
    (y_0 : y_1 : y_2) \mapsto (y_0^3 : y_0y_1^2 : y_0^2y_1 : y_0^2y_2 : -(y_0y_2^2 + y_1^3))\text,
  \end{equation*}
  and the map $\Psi$ from \cite[Claim 4.2]{arXiv:1302.6151} sending $(\e_1,
  \ldots, \e_9)$ to
  \begin{equation*}
    (\e_1^6\e_2^5\e_3^3\e_4^4\e_5^2\e_6^4, \e_1^2\e_2\e_3\e_4^2\e_5^2\e_7^2,
    \e_1^4\e_2^3\e_3^2\e_4^3\e_5^2\e_6^2\e_7,\e_1^3\e_2^2\e_3^2\e_4^2\e_5\e_6\e_8, \e_9),
  \end{equation*}
  we see that the requirements of \cite[Lemma 4.3]{arXiv:1302.6151} are
  satisfied, so \cite[Claim 6.2]{arXiv:1302.6151} holds for $i=0$.

  As in the proof of Lemma~\ref{lem:A4_passage_to_torsor}, we apply
  \cite[Remark~4.5]{arXiv:1302.6151} for steps (1), (2), (3).
  For (1), we define $\e_5'' \in C_1$ with $[I_3'+I_7'+I_9']=[C_1^{-1}]$ such
  that $I_5''=I_3'+I_7'+I_9'$. We must use the relation
  $\e_3''\e_8''^2+\e_9''+\e_5''^2\e_7''^3=0$ to check the coprimality
  conditions for $\e_3'', \e_5'', \e_7'', \e_9''$, namely
  $\eI_3''+\eI_7''=\OO_K$ (this holds because of the relation and
  $\eI_3''+\eI_7''+\eI_9''=\OO_K$ by construction) and
  $\eI_5''+\eI_7''+\eI_9''=\OO_K$ (this holds because of the relation and
  $\eI_3''+\eI_7''+\eI_9''=\OO_K$ by construction and the coprimality
  condition $I_5''+I_8''=\OO_K$ provided by the proof of
  \cite[Lemma~4.4]{arXiv:1302.6151}).

  For (2), we define $\e_4'' \in C_2$ with $[I_3'+I_5'+I_9']=[C_2^{-1}]$ such
  that $I_4''=I_3'+I_5'+I_9'$. The relation is
  $\e_3''\e_8''^2+\e_9''+\e_4''\e_5''^2\e_7''^3=0$. We check the coprimality
  conditions $I_3''+I_5''=\OO_K$ (this holds because of the relation and
  $I_3''+I_5''+I_9''=\OO_K$ by construction) and $I_5''+I_9''=\OO_K$ (this
  holds because of the relation and $I_3''+I_5''=\OO_K$ as just shown and
  $I_5''+I_8''=\OO_K$ by the proof of \cite[Lemma~4.4]{arXiv:1302.6151}).

  For (3), we define $\e_1'' \in C_3$ with $[I_3'+I_4'+I_9']=[C_3^{-1}]$ such
  that $I_1''=I_3'+I_4'+I_9'$. The relation is
  $\e_3''\e_8''^2+\e_9''+\e_4''\e_5''^2\e_7''^3=0$. We check the coprimality
  conditions are $I_3''+I_4''=\OO_K$ (this holds because of the relation and
  $I_3''+I_4''+I_9''=\OO_K$ by construction), $I_3''+I_9''=\OO_K$ (this holds
  because of the relation and $I_3''+I_4''=\OO_K$ as just shown
  $I_3''+I_5''=\OO_K$ as before and $I_3''+I_7''=\OO_K$ as before) and
  $I_4''+I_9''=\OO_K$ (this holds because of the relation and
  $I_3''+I_9''=\OO_K$ as just shown and $I_4''+I_8''=\OO_K$ as before).

  For (4) and (5), we can apply \cite[Lemma~4.4]{arXiv:1302.6151}. This proves
  \cite[Claim 6.2]{arXiv:1302.6151}, and we deduce \cite[Claim
  6.1]{arXiv:1302.6151} as in \cite[Lemma 9.1]{arXiv:1302.6151}.
\end{proof}

\subsection{Summations}
\subsubsection{The first summation over $\e_8$ with dependent $\e_9$}
Let $\ee' := (\e_1, \ldots, \e_7)$ and $\eII' := (\eI_1, \ldots, \eI_7)$. Let
$\theta_0(\eII') := \prod_\p\theta_{0,\p}(J_\p(\eII'))$, where $J_\p(\eII') :=
\{j \in \{1, \ldots, 7\}\ :\  \p \mid \eI_j\}$ and
\begin{equation*}
  \theta_{0,\p}(J):=
  \begin{cases}
    1 &\text{ if } J = \emptyset, \{1\}, \{2\}, \{3\}, \{4\}, \{5\}, \{6\},
    \{7\},\\
    \ &\text{ or } J= \{1,2\}, \{1,3\}, \{1,4\}, \{2,6\}, \{4,5\}, \{5,7\},\\
    0 &\text{ otherwise.} 
  \end{cases}
\end{equation*}

We apply \cite[Proposition 5.3]{arXiv:1302.6151} with $(A_1, A_2, A_0) :=
(4,5,7)$, $(B_1, B_0) := (3, 8)$, $(C_1, C_2, C_0):=(2,6,9)$, and $D := 1$. For
given $\e_1$, $\e_2$, $\e_6$, we write
\begin{equation*}
  \e_4\e_5^2\e_7^3 = \e_{A_0}^{a_0}\Pi(\eeA) = \Ao \At^2,
\end{equation*}
where $\Ao$, $\At$ are chosen as follows: Let $\id A = \id A(\e_1, \e_2, \e_6)$ be a prime ideal not
dividing $\eI_1\eI_2\eI_6$ such that $\id A\OO_7^{-1} \OO_8 = \id A C_1$
is a principal fractional ideal $t\OO_K$, for a suitable $t = t(\e_1, \e_2, \e_6) \in
K^\times$. Then we define $\At = \At(\e_1, \e_2, \e_6) := \e_7t$ and $\Ao := \Ao(\e_1,
\e_2, \e_6) := \e_4\e_5^2\e_7t^{-2}$.

\begin{lemma}\label{lem:D5_first_summation}
  We have
  \begin{equation*}
    |M_\classtuple(B)| = \frac{2}{\sqrt{|\Delta_K|}}\sum_{\ee' \in \OO_{1*} \times \dots \times \OO_{7*}} \theta_8(\ee', \classtuple)V_8(\N\eI_1,\ldots,\N\eI_7; B) + O_\classtuple(B),
  \end{equation*}
  where
  \begin{equation*}
    V_8(t_1, \ldots, t_7; B) :=  \frac{1}{t_2 t_6^2}\int_{(\sqrt{t_1}, \ldots,
      \sqrt{t_7}, \e_8)\in\mathcal{R}(B)} \dd \e_8,
  \end{equation*}
  with a complex variable $\e_8$. Moreover,
  \begin{equation*}
    \theta_8(\ee', \classtuple) := \sum_{\substack{\kc \mid \eI_1\eI_2\\\kc + \eI_3\eI_4 =
        \OO_K}}\frac{\mu_K(\kc)}{\N\kc}\tilde\theta_8(\eII',\kc)\sum_{\substack{\rho
        \mod \kc\eI_2\eI_6^2\\\rho\OO_K + \kc\eI_2\eI_6^2 = \OO_K \\\rho^2 \equiv_{\kc\eI_2\eI_6^2}         \e_7 A}}1\text,
  \end{equation*}
  with
  \begin{equation*}
    \tilde\theta_8(\eII', \kc) :=
    \theta_0(\eII')\frac{\phi_K^*(\eI_1\eI_4\eI_5)}{\phi_K^*(\eI_1 + \kc\eI_2\eI_6)}\text.
  \end{equation*}
  Here, $A := -\e_4\e_5^2/(t(\e_1,\e_2,\e_6)^2\e_3)$, and $\e_7 A$ is
  invertible modulo $\kc\eI_2\eI_6^2$ whenever $\theta_0(\eII')\neq 0$.
\end{lemma}

\begin{proof}
  As in Lemma \ref{lem:A4_first_summation}, we see that the main term from
  \cite[Proposition 5.3]{arXiv:1302.6151} is the one given in the lemma. Let us
  consider the error term. For given $\ee'$, the set of all $\e_8$ with $(\e_1,
  \ldots, \e_8) \in \mathcal{R}(u_\classtuple B)$ has bounded class and is
  contained in two balls of radius
  \begin{equation*}
    R(\ee'; u_\classtuple B) \ll
    \begin{cases}
       B^{3/8}\N(\eI_2^{3}\eI_3^{-2}\eI_4^{-1}\eI_5^{-2}\eI_6^{6}\eI_7^{-3})^{1/8}&\text{ if }\e_7\neq 0\\
      (B/\N(\eI_1^3\eI_2^2\eI_3^2\eI_4^2\eI_5\eI_6))^{1/2} &\text{ if
      }\e_7 = 0\text.
    \end{cases}
  \end{equation*}
  If $\eta_7 \neq 0$ this follows from taking the geometric mean of both
  expressions in the minimum in \cite[Lemma 3.5, (1)]{arXiv:1302.6151} applied to
  \eqref{eq:D5_height_5}. If $\eta_7 = 0$ then it follows from
  \eqref{eq:D5_height_4}.
  Thus, the error term is
  \begin{equation*}
    \ll \sum_{\ee'\text{, }\eqref{eq:D5_first_height_cond_ideals}\text{, }\eqref{eq:D5_third_height_cond_ideals}}2^{\omega_K(\eI_1\eI_2)+\omega_K(\eI_1\eI_4\eI_5)+\omega_K(\eI_1\eI_2\eI_6)}\left(\frac{R(\ee'; u_\classtuple B)}{\N(\eI_2 \eI_6^2)^{1/2}}+1\right)\text,
  \end{equation*}
  where, using \eqref{eq:D5_height_1} and \eqref{eq:D5_height_3}, the sum runs over all $\ee' \in \OO_{1*}\times\cdots\times\OO_{7*}$ with
  \begin{align}\label{eq:D5_first_height_cond_ideals}
    \N(\eI_1^6\eI_2^5\eI_3^3\eI_4^4\eI_5^2\eI_6^4) &\leq B \text{, and }\\
    \N(\eI_1^4\eI_2^3\eI_3^2\eI_4^3\eI_5^2\eI_6^2\eI_7) &\leq B
    \text.\label{eq:D5_third_height_cond_ideals}
  \end{align}

  The sum over all $\ee'$ with $\e_7 \neq 0$ is bounded by
  \begin{align*}
    &\ll_\classtuple \sum_{\eII'\text{, }\eqref{eq:D5_third_height_cond_ideals}}2^{\omega_K(\eI_1\eI_2)+\omega_K(\eI_1\eI_4\eI_5)+\omega_K(\eI_1\eI_2\eI_6)}\left(\frac{B^{3/8}}{(\N\eI_2\N\eI_3^2\N\eI_4\N\eI_5^2\N\eI_6^2\N\eI_7^3)^{1/8}}+1\right)\\
    &\ll\sum_{\substack{\eI_1, \dots, \eI_6\\\N\eI_j\leq B}}\left(\frac{2^{\omega_K(\eI_1\eI_2)+\omega_K(\eI_1\eI_4\eI_5)+\omega_K(\eI_1\eI_2\eI_6)}B}{\N\eI_1^{5/2}\N\eI_2^2\N\eI_3^{3/2}\N\eI_4^2\N\eI_5^{3/2}\N\eI_6^{3/2}}+\frac{2^{\omega_K(\eI_1\eI_2)+\omega_K(\eI_1\eI_4\eI_5)+\omega_K(\eI_1\eI_2\eI_6)}B}{\N\eI_1^4\N\eI_2^3\N\eI_3^2\N\eI_4^3\N\eI_5^2\N\eI_6^2}\right)\\
    &\ll B.
  \end{align*}
  The sum over all $\ee'$ with $\e_7 = 0$ is bounded by
  \begin{align*}
    &\ll_\classtuple \sum_{\substack{\eI_1, \ldots, \eI_6\\\eqref{eq:D5_first_height_cond_ideals}}}2^{\omega_K(\eI_1\eI_2)+\omega_K(\eI_1\eI_4\eI_5)+\omega_K(\eI_1\eI_2\eI_6)}\left(\frac{B^{1/2}}{(\N\eI_1^3\N\eI_2^3\N\eI_3^2\N\eI_4^2\N\eI_5\N\eI_6^3)^{1/2}}+1\right)\\
    &\ll\sum_{\substack{\eI_1, \dots, \eI_4, \eI_6\\\N\eI_j \leq 1}}2^{\omega_K(\eI_1\eI_2)}\left(\frac{2^{\omega_K(\eI_1\eI_4\eI_5)+\omega_K(\eI_1\eI_2\eI_6)}B^{3/4}}{\N\eI_1^{3}\N\eI_2^{11/4}\N\eI_3^{7/4}\N\eI_4^2\N\eI_6^{5/2}}+\frac{2^{\omega_K(\eI_1\eI_4\eI_5)+\omega_K(\eI_1\eI_2\eI_6)}B^{1/2}}{\N\eI_1^3\N\eI_2^{5/2}\N\eI_3^{3/2}\N\eI_4^2\N\eI_6^2}\right)\\
    &\ll B^{3/4} + B^{1/2} \ll B^{3/4}.\qedhere
  \end{align*}
\end{proof}

\subsubsection{The second summation over $\e_7$}
We define  
 \begin{equation*}
    \theta_8'(\eII') := \sum_{\substack{\kc \mid \eI_1\eI_2\\\kc + \eI_3\eI_4 =
        \OO_K}}\frac{\mu_K(\kc)}{\N\kc}\tilde\theta_8(\eII',\kc).
  \end{equation*}

\begin{lemma}\label{lem:D5_second_summation}
  Write $\ee'' := (\e_1, \dots, \e_6)$. We have
  \begin{align*}
    |M_\classtuple(B)| &=\left(\frac{2}{\sqrt{|\Delta_K|}}\right)^2 \sum_{\ee''\in \OO_{1*} \times \dots \times \OO_{6*}} \mathcal{A}(\theta_8'(\eII'),\eI_7)V_7(\N\eI_1,\ldots,\N\eI_6; B)\\ &+ O_\classtuple(B(\log B)^3),
  \end{align*}
  where, for $t_1,\ldots,t_6 \geq 1$,
  \begin{equation*}
    V_7(t_1, \ldots, t_6; B) := \frac{\pi}{t_2 t_6^2}\int_{(\sqrt{t_1}, \ldots, \sqrt{t_7}, \e_8)\in\mathcal{R}(B)} \dd t_7 \dd \e_8,
  \end{equation*}
  with a positive variable $t_7$ and a complex variable $\e_8$.
\end{lemma}

\begin{proof}
  We write
  \begin{equation*}
    |M_\classtuple(B)| = \frac{2}{\sqrt{|\Delta_K|}}\sum_{\ee'' \in \OO_{1*} \times \dots \times \OO_{6*}} \sum_{\substack{\kc \mid \eI_1\eI_2\\\kc+ \eI_3\eI_4 = \OO_K}}\frac{\mu_K(\kc)}{\N\kc}\Sigma + O_\classtuple(B),
  \end{equation*}
  where
  \begin{equation*}
    \Sigma := \sum_{\e_7 \in
      \OO_{7}}\vartheta(\eI_7)\sum_{\substack{\rho \mod
        \kc\eI_2\eI_6^2\\\rho\OO_K + \kc\eI_2\eI_6^2=\OO_K \\\rho^2 \equiv_{\kc\eI_2\eI_6^2} \e_7 A}}g(\N\eI_7)\text,
  \end{equation*}
  with $\vartheta(\eI_7) := \tilde\theta_8(\eII',\kc)$ and $g(t):=
  V_8(\N\eI_1,\ldots,\N\eI_6, t; B)$. As in Lemma
  \ref{lem:A4_second_summation_Ma}, the function $\vartheta$ satisfies
  \cite[(6.1)]{arXiv:1302.6151} with $C:=0$,
  $c_\vartheta:=2^{\omega_K(\eI_1\eI_2\eI_3\eI_4\eI_6)}$.  Moreover, by
  \eqref{eq:D5_height_2}, we have $g(t) = 0$ whenever $t > t_2 :=
  B^{1/2}/(\N\eI_1 \N\eI_2^{1/2}\N\eI_3^{1/2}\N\eI_4\N\eI_5)$, and by
  \cite[Lemma 3.5, (2)]{arXiv:1302.6151} applied to \eqref{eq:D5_height_5}, we
  obtain $g(t) \ll B^{1/2}/(\N\eI_2^{1/2}\N\eI_3^{1/2}\N\eI_6)$.

  By \cite[Proposition 6.1]{arXiv:1302.6151}, we have
  \begin{align*}
    \Sigma &= \vartheta((0))g(0) + \frac{2 \pi}{\sqrt{|\Delta_K|}}\phi_K^*(\kc
    \eI_2\eI_6^2)\mathcal{A}(\vartheta(\aaa), \aaa, \kc\eI_2\eI_6^2)\int_{t\geq
      1}g(t)\dd t\\ &+ O\left(\frac{2^{\omega_K(\eI_1\eI_2\eI_3\eI_4\eI_6)}
        B^{1/2}}{\N\eI_2^{1/2}\N\eI_3^{1/2}\N\eI_6}\left(\frac{B^{1/4}\N\kc^{1/2}\N\eI_2^{1/4}\N\eI_6}{\N\eI_1^{1/2}\N\eI_3^{1/4}\N\eI_4^{1/2}\N\eI_5^{1/2}}
        + \N(\kc\eI_2\eI_6^2)\log B\right)\right)\text,
  \end{align*}
  where the contribution of $\vartheta((0))g(0)$ and of $\int_0^1g(t)\dd t$ is
  dominated by the error term.  Using \cite[Lemma 6.3]{arXiv:1302.6151}, we
  see that this gives the correct main term in the lemma.

  Summing the error term over $\eI_j$ and $\kc$, we obtain an upper bound
  \begin{align*}
    &\sum_{\substack{\eI_1, \ldots, \eI_6\\\eqref{eq:D5_first_height_cond_ideals}}}\left(\frac{2^{\omega_K(\eI_1\eI_2) + \omega_K(\eI_1\cdots\eI_4\eI_6)} B^{3/4}}{\N\eI_1^{1/2}\N\eI_2^{1/4}\N\eI_3^{3/4}\N\eI_4^{1/2}\N\eI_5^{1/2}} + \frac{2^{\omega_K(\eI_1\eI_2) + \omega_K(\eI_1\cdots\eI_4\eI_6)}B^{1/2}\log B}{\N\eI_2^{-1/2}\N\eI_3^{1/2}\N\eI_6^{-1}} \right)\\
    &\ll \sum_{\substack{\eI_1, \ldots, \eI_5\\\N\eI_j \leq B}}\left(\frac{2^{\omega_K(\eI_1\eI_2) + \omega_K(\eI_1\cdots\eI_4)}B\log B}{\N\eI_1^2\N\eI_2^{3/2}\N\eI_3^{3/2}\N\eI_4^{3/2}\N\eI_5} + \frac{2^{\omega_K(\eI_1\eI_2) + \omega_K(\eI_1\cdots\eI_4)}B(\log B)^2}{\N\eI_1^3\N\eI_2^2\N\eI_3^2\N\eI_4^2\N\eI_5}\right)\\
    &\ll B(\log B)^2 + B(\log B)^3 \ll B(\log B)^3\text.\qedhere
  \end{align*}
\end{proof}

\begin{lemma}\label{lem:D5_second_summation_ideals}
  If $\eII''$ runs over all six-tuples $(\eI_1, \ldots, \eI_6)$ of
  nonzero ideals of $\OO_K$ then we have
  \begin{equation*}
    N_{U_4,H}(B) = \left(\frac{2}{\sqrt{|\Delta_K|}}\right)^2\sum_{\eII''}\mathcal{A}(\theta_8'(\eII'),\eI_7)V_7(\N\eI_1, \ldots, \N\eI_6; B) + O(B(\log B)^4)\text. 
  \end{equation*}
\end{lemma}

\begin{proof}
  This is analogous to \cite[Lemma 9.4]{arXiv:1302.6151}.
\end{proof}

\subsubsection{The remaining summations}
\begin{lemma}\label{lem:D5_completion}
  We have
  \begin{equation*}
    N_{U_4,H}(B) = \left(\frac{2}{\sqrt{|\Delta_K|}}\right)^8 \left(\frac{h_K}{\omega_K}\right)^6 \theta_0V_0(B) + O(B(\log B)^4\log \log B),
  \end{equation*}
  where $\theta_0$ is as in \eqref{eq:def_theta_0} and
  \begin{equation*}
    V_0(B) := \int\limits_{\substack{(\e_1, \ldots, \e_8)\in\mathcal{R}(B)\\\abs{\e_1}, \ldots, \abs{\e_6} \ge 1}}\frac{1}{\abs{\e_2 \e_6^2}}\dd \e_1 \cdots \dd \e_8,
  \end{equation*}
 with complex variables $\e_1, \ldots, \e_8$.
\end{lemma}

\begin{proof}
  By \cite[Lemma 3.5, (5)]{arXiv:1302.6151} applied to \eqref{eq:D5_height_5}, we have
  \begin{equation*}
    V_7(t_1, \ldots, t_6; B) \ll \frac{B^{5/6}}{t_2^{1/6}t_3^{1/2}t_4^{1/3}t_5^{2/3}t_6^{1/3}}
    =\frac{B}{t_1 \cdots
      t_6}\left(\frac{B}{t_1^6t_2^5t_3^3t_4^4t_5^2t_6^4}\right)^{-1/6}\text.
  \end{equation*}
  We apply \cite[Proposition 7.3]{arXiv:1302.6151} with $r=5$ and use polar coordinates.
\end{proof}

\subsection{Proof of Theorem \ref{thm:main} for $S_4$}
\begin{lemma}\label{lem:D5_predicted_volume}
  Let $\alpha(\tS_4)$, $\omega_\infty(\tS_4)$ be as in Theorem \ref{thm:main},
  $\mathcal{R}(B)$ as in \eqref{eq:D5_height_1}--\eqref{eq:D5_height_5}, and define
  \begin{equation*}
    V_0'(B) := \int_{\substack{(\e_1, \ldots, \e_8) \in
        \mathcal{R}(B)\\\abs{\e_1}\text{, }\abs{\e_2}\text{, }\abs{\e_4}\text{,
        }\abs{\e_5}\text{, }\abs{\e_6}\geq
        1\\\abs{\e_1^6\e_2^5\e_4^4\e_5^2\e_6^4} \le
        B}}\frac{1}{\abs{\e_2\e_6^2}}\dd \e_1 \cdots \dd \e_8,
  \end{equation*}
  where $\e_1, \ldots, \e_8$ are complex variables. Then
  \begin{equation}\label{eq:D5_predicted_volume}
    \pi^6\alpha(\tS_4) \omega_\infty(\tS_4) B(\log B)^5 = 4 V_0'(B).
  \end{equation}
\end{lemma}
\begin{proof}
  Let $\e_1,\e_2,\e_4,\e_5,\e_6 \in \CC$, $B>0$, and $l := (B
  \abs{\e_1^3\e_2^4\e_4^2\e_5\e_6^5})^{1/2}$. Let $\e_3, \e_7, \e_8$ be complex variables. We apply the coordinate
  transformation $z_0 = l^{-1/3}\e_1^3\e_2^3\e_4^2\e_5\e_6^3\cdot \e_3$,
  $z_2 = l^{-1/3}\e_1\e_2\e_4\e_5\e_6\cdot \e_7$, $z_3 =
  l^{-1/3}\cdot \e_8$ to $\omega_\infty(\tS_4)$ and obtain
  \begin{equation}\label{eq:D5_complex_density_torsor}
    \omega_\infty(\tS_4) = \frac{12}{\pi}\frac{\abs{\e_1\e_2\e_4\e_5\e_6}}{B}
    \int_{(\e_1, \ldots, \e_8)\in\mathcal{R}(B)}\frac{1}{\abs{\e_2\e_6^2}}\dd \e_3 \dd \e_7 \dd \e_8\text.
  \end{equation}

  Since the negative curves $[E_1], \dots, [E_6]$ generate the effective
  cone of $\tS_4$, and $[-K_{\tS_4}] = [6E_1+5E_2+3E_3+4E_4+2E_5+4E_6]$,
  \cite[Lemma 8.1]{arXiv:1302.6151} gives
  \begin{equation}\label{eq:D5_alpha}
      \alpha(\tS_4)(\log B)^5= \frac 1{3\pi^5} \int_{\substack{\abs{\e_1}\text{, }\abs{\e_2}\text{, }\abs{\e_4}\text{,
        }\abs{\e_5}\text{, }\abs{\e_6}\geq
        1\\\abs{\e_1^6\e_2^5\e_4^4\e_5^2\e_6^4} \le
        B}} \frac{\dd \e_1\dd \e_2\dd \e_4\dd \e_5\dd
        \e_6}{\abs{\e_1\e_2\e_4\e_5\e_6}}\text.
  \end{equation}
  The lemma follows by substituting
  \eqref{eq:D5_complex_density_torsor} and \eqref{eq:D5_alpha} in
  \eqref{eq:D5_predicted_volume}.
\end{proof}

To finish our proof, we compare $V_0(B)$ defined in Lemma
\ref{lem:D5_completion} with $V_0'(B)$ defined in Lemma
\ref{lem:D5_predicted_volume}. Starting from $V_0(B)$,
we can add the condition $\abs{\e_1^6\e_2^5\e_4^4\e_5^2\e_6^4} \le B$ and remove $\abs{\e_3} \ge
1$ with negligible error. Indeed, adding the condition $\abs{\e_1^6\e_2^5\e_4^4\e_5^2\e_6^4} \le B$ to the domain of
integration for $V_0(B)$ does not change the result. Using \cite[Lemma 3.5,
(3)]{arXiv:1302.6151} applied to \eqref{eq:D5_height_5} to bound the integral
over $\e_7, \e_8$, we see that $V_0'(B)-V_0(B)$ is 
\begin{equation*}
  \ll \int_{\substack{\abs{\e_1},\abs{\e_2}, \abs{\e_4}, \abs{\e_5}, \abs{\e_6} \ge 1\\\abs{\e_3} < 1\\\abs{\e_1^6\e_2^5\e_4^4\e_5^2\e_6^4} \le B}} \frac{B^{5/6}}{\abs{\e_2\e_3^3\e_4^2\e_5^4\e_6^2}^{1/6}}\dd \e_1\cdots\dd \e_6 \ll B(\log B)^4.
\end{equation*}
Using Lemma~\ref{lem:D5_completion} and Lemma~\ref{lem:D5_predicted_volume},
this implies Theorem \ref{thm:main} for $S_4$.

\bibliographystyle{alpha}

\bibliography{counting_imaginary_quadratic_points_2}

\end{document}